\documentclass[11pt]{article}
\usepackage{fullpage}
\usepackage{graphicx}
\usepackage{graphics}
\usepackage{psfrag}
\usepackage{amsmath,amssymb}
\usepackage{enumerate}
\usepackage{amsthm}
\usepackage{amsfonts}
\usepackage{bbm}
\usepackage{hyperref}
\usepackage{cleveref}
\usepackage{tikz-cd}
\usepackage{multicol}

\newcommand{\Z}{\mathbb{Z}}

\newcommand{\GEN}{\texttt{GEN}}
\newcommand{\DNG}{\texttt{DNG}}

\newcommand{\RAV}{\texttt{RAV}}
\newcommand{\REL}{\texttt{REL}}

\DeclareMathOperator{\dic}{Dic}

\newcommand{\Addresses}{{
  \bigskip
  \footnotesize
    Zachary Gates, \textsc{Department of Mathematics \& Computer Science, Wabash College, Crawfordsville, IN 47933}\par\nopagebreak
  \textit{E-mail address}: \texttt{gatesz@wabash.edu}
  
  \medskip
  
  Robert Kelvey, \textsc{Department of Mathematical \& Computational Sciences, The College of Wooster,
    Wooster, OH 44691}\par\nopagebreak
  \textit{E-mail address}: \texttt{rkelvey@wooster.edu}

}}

\theoremstyle{definition}
\newtheorem{theorem}{Theorem}[section]
\newtheorem{definition}[theorem]{Definition}
\newtheorem{example}[theorem]{Example}

\newtheorem{lemma}[theorem]{Lemma}
\newtheorem{proposition}[theorem]{Proposition}

\newtheorem{remark}[theorem]{Remark}

\renewcommand{\cref}{\Cref}

\title{Relator Games on Groups}
\author{Zachary Gates and Robert Kelvey}
\date{}
\begin{document} 

\maketitle

\begin{abstract}
    We define two impartial games, the \textit{Relator Achievement Game} $\REL$ and the \textit{Relator Avoidance Game} $\RAV$. Given a finite group $G$ and generating set $S$, both games begin with the empty word. Two players form a word in $S$ by alternately appending an element from $S\cup S^{-1}$ at each turn. The first player to form a word equivalent in $G$ to a previous word wins the game $\REL$ but loses the game $\RAV$. Alternatively, one can think of $\REL$ and $\RAV$ as \textit{make a cycle} and \textit{avoid a cycle} games on the Cayley graph $\Gamma(G,S)$. We determine winning strategies for several families of finite groups including dihedral, dicyclic, and products of cyclic groups. 
\end{abstract}

\section{Introduction}\label{sec: introduction}
In this paper we define two $2$-player combinatorial games: the \textit{Relator Achievement Game} $\REL$ and the \textit{Relator Avoidance Game} $\RAV$. Given a finite group $G$ and generating set $S$, two players take turns choosing $s$ or $s^{-1}$, where $s$ is a generator from $S$. The only stipulation is that, if the previous player chose $s$, the next player cannot choose $s^{-1}$ and vice versa. The players' choice of group elements builds a word in $S$. The goal of $\REL$ is to be the first player to achieve a subword equivalent to the identity in $G$. The game of $\RAV$ is the mis\`{e}re version of $\REL$, meaning the first player to achieve a subword equivalent to the identity loses the game. One can play these games on the Cayley graph of $G$ formed by using the generating set $S$. Since paths in a Cayley graph correspond to words in $S$, the players' choices of generators form a path in the Cayley graph without backtracking. Hence, when viewed graphically, the goal of $\REL$ is to be the first player to make a cycle, whereas for $\RAV$ the goal is to avoid cycles.

One motivation for the development of the games $\REL$ and $\RAV$ originated from recent results by Benesh, Ernst, and Sieben in \cite{BES16-symalt, BES16-1stPub, BES17, BES19} for the combinatorial games $\GEN$ and $\DNG$, which were first defined by Anderson and Harary in \cite{AH87}. In these games, two players alternate choosing distinct elements from a finite group $G$ until $G$ is generated by the chosen elements. The first player to generate the group on their turn wins the game $\GEN$, but loses $\DNG$. Taking inspiration from this work, our goal was to create a pair of games that incorporates the geometry of a group $G$ through its Cayley graph. However, $\REL$ and $\RAV$ are distinct from the combinatorial games involving graphs found in the current literature.
For example, \textit{Cops and Robbers} (see \cite{NW83}, \cite{Quill78}) a popular pursuit-evasion game, has been studied specifically on Cayley graphs (see \cite{Frankl87}), and firefighting games have been studied on Cayley graphs as well (see \cite{Leh19}). More recently, the \textit{Game of Cycles} was introduced by Su in \cite{Su20} and expounded in \cite{cycleGame} by Alvarado, et al. This game involves planar graphs and two players taking turns marking previously unmarked edges with a chosen direction. The \textit{Game of Cycles} is the closest of these combinatorial games to $\REL$ and $\RAV$, since the goal of the game is to create a cycle. However, the parameters for doing so are very different than in our game of $\REL$.

This paper is structured as follows. In \cref{sec: two-player rel and rav} we give a precise definition (\cref{def: REL and RAV games}) of the games $\REL$ and $\RAV$ along with some examples and initial results concerning complete bipartite and complete Cayley graphs (see \cref{thm:complete-bipartite} and \cref{thm:complete-graph}). 
In \cref{sec: dihedral groups}, we explore the family of dihedral groups $D_n$, $n \geq 3$, with its canonical generating sets. We show winning strategies for the game $\REL$ in \cref{thm: 2 player dihedral REL} and for the game $\RAV$ in \cref{thm: 2-player rav dihedral}. At the end of this section, we generalize the result of \cref{thm: 2-player rav dihedral} to apply to any group with a generating set including an element of order $2$ in \cref{thm: RAV any group with order two}.
In \cref{sec: dicyclic}, we explore the family of dicyclic groups with two common generating sets. These are results \cref{thm: RAV for Dic with two generators}, \cref{thm: RAV for dic with three generators}, \cref{thm: REL for odd dicyclic}, and \cref{thm: REL for dic even}.
In \cref{sec: products of groups}, we examine $\REL$ for products of cyclic groups $\Z_n \times \Z_m$, where the results depend on $n$ modulo $m$ (\cref{theorem: Zn-Zm mod pm1 and 0}). 
In \cref{sec: three-player REL} we introduce the game $\REL_n$, which is the $n$-player version of $\REL$, and prove a winning strategy for $\REL_3$ on the dihedral groups $D_n$ in \cref{thm:3-player REL}.
Lastly, we conclude with some open questions in \cref{sec: open questions}.

\section{Two-Player Relator Games \REL ~and \RAV} \label{sec: two-player rel and rav}

In \cite{BES16-symalt, BES16-1stPub, BES17, BES19}, Benesh, Ernst, and Sieben analyze the games $\GEN$ (generate a group) and $\DNG$ (do not generate) as initially proposed by Anderson and Harary in \cite{AH87}. Given a finitely generated group $G$, two players alternate choosing elements from $G$ and pool the elements together to generate a subgroup of $G$. The first to have chosen a generating set for the whole group on their turn wins the game of $\GEN$. For example, if $G$ is a cyclic group, then Player 1 can easily win the game of $\GEN$ by choosing a generator. In the game of $\DNG$, the two players are trying to \textit{not} generate the group $G$. Instead of two players choosing elements from $G$ and then generating a subgroup, the games $\REL$ and $\RAV$ we have two players select elements from a generating set $S$ of $G$, forming an ever-growing word in $G$. 

Let $G$ be a finite group and let $S$ be a generating set for $G$ (with $e \notin S$). We define two two-player impartial combinatorial games, called the \textit{Relator Achievement Game} $\REL(G,S)$, and the \textit{Relator Avoidance Game} $\RAV(G,S)$, as follows.

\begin{definition}\label{def: REL and RAV games}
        On turn 1, Player 1 begins with the empty word $w_0$. Player 1 chooses an element $s_1\in S\cup S^{-1}$ to create the word $w_1 = w_0s_1=s_1$. The players then alternate choosing elements of $S\cup S^{-1}$. On turn $n$, with $n > 1$, the current player begins with a word 
        \[
            w_{n-1} = s_1 s_2 \dots s_{n-1}.
        \]
    
           They then select a generator $s_n \in S \cup S^{-1}$ such that $s_n \neq s_{n-1}^{-1}$ and form the word $w_n$:
       \[
        w_n = w_{n-1}s_n.
       \]
       
        If a player forms $w_n$ such that $w_n \equiv_G w_k$, that is, $w_n$ and $w_k$ represent the same element of $G$, for some $0\leq k<n$, then that player wins $\REL(G,S)$ and loses $\RAV(G,S)$, respectively.
        
        If from any position there are no legal moves, then the next player loses. Otherwise, play passes to the next player and continues as described above.
\end{definition}

\begin{remark}
When the group and generating set are clear from context, we will use the shorthand $\REL$ or $\RAV$ to refer to the Relator Achievement Game or the Relator Avoidance Game, respectively, for a group $G$ and generating set $S$.
We forbid the trivial relator $ss^{-1}$ in our games since every group contains these relators, and we are seeking non-trivial relators. We also assume in our definition that a generating set $S$ does not contain the identity for similar reasons.

For the trivial group and the cyclic group of order $2$, with their canonical generating sets, both games end due to the eventual absence of a legal move. These are in fact the only groups where this occurs.
\end{remark}

Recall that the Cayley graph $\Gamma(G,S)$ for a group $G$ and generating set $S$ is a  graph with vertices the elements of $G$ and a directed edge from vertex $g$ to vertex $h$ if $h=gs$ for some $s \in S$. Such an edge would be labeled by $s$. 

If we consider a path of edges in the Cayley graph $\Gamma(G,S)$, this will correspond to a word $w = s_1 s_2 \dots s_{n-1}s_n$ in $G$ with letters in $S \cup S^{-1}$. Therefore, one can visually play the games of $\REL$ and $\RAV$ on a Cayley graph: a players' choice of element $s_n \in S \cup S^{-1}$
will correspond to traversing an undirected edge in the Cayley graph. A player wins $\REL$ if they are the first to form a cycle (a relator) in the Cayley graph. Likewise, a player loses $\RAV$ if they are the first to form a cycle in the Cayley graph. The rule stating that a player may not choose the inverse of the last generator chosen translates to disallowing backtracking in the Cayley graph.

We mention this visual Cayley graph correspondence as a useful way to analyze the games $\REL$ and $\RAV$. It can be helpful to play these games on a Cayley graph to understand the winning strategies for different groups and generating sets. Note that, due to how players choose elements from $S \cup S^{-1}$, whenever we discuss Cayley graphs, we refer to the undirected Cayley graph.

\begin{example}\label{ex: REL-RAV cyclic}
Consider $\REL(\Z_n, \{1\})$, where $\Z_n$ denotes the additive group of integers modulo $n$ with $n>2$.
The corresponding Cayley graph for $(\Z_n, \{1\})$ is an $n$-sided polygon. Hence, the games $\REL(\Z_n, \{1\})$ and $\RAV(\Z_n, \{1\})$ are completely determined by the parity of $n$. If $n$ is even, then Player 2 will win $\REL$ and Player 1 will win $\RAV$.
If $n$ is odd, then Player 1 will win $\REL$ and Player 2 will win $\RAV$.
\end{example}

\begin{example} \label{ex: quaternions}
  Consider the quaternion group $Q_8$ with generating set $S=\{i,j\}$. We can investigate the game $\REL(Q_8, S)$ by means of the Cayley graph $\Gamma(Q_8, S)$. See \cref{fig:Q8}, where the labels $i$ and $j$ are denoted by \textcolor{blue}{blue} and \textcolor{red}{red}, respectively. Note that $\Gamma(Q_8,S)$ is a complete, bipartite graph. That is, the vertices can be partitioned into two sets $A$ and $B$ such that for any two vertices $a\in A$ and $b\in B$, there is an edge joining $a$ and $b$, and for any two elements from the same set, there is no edge between them.
  In this case, we have $A=\{\pm 1, \pm k\}$ and $B = \{\pm i, \pm j\}$. These sets of vertices are shaded \textcolor{cyan}{cyan} and \textcolor{orange}{orange}, respectively, in \cref{fig:Q8}.
  
To see who wins $\REL(Q_8, S)$, note that Player 1 must choose from the set $B$ on their first turn. Player 2 cannot backtrack to $1$ and so must move to a vertex from $A - \{1\}$. Then, Player 1 moves to another vertex from $B$ distinct from their previous choice and thus cannot win on this turn. Finally, Player 2 wins on their second turn by moving back to $1$.
\end{example}

\tikzstyle{vert-A} = [circle, draw, inner sep=0pt, minimum size=6.5mm, fill=cyan]
\tikzstyle{vert-B} = [circle, draw, inner sep=0pt, minimum size=6.5mm, fill=orange]
\tikzstyle{b} = [draw,very thick,blue]
\tikzstyle{r} = [draw, very thick, red]
\tikzstyle{g} = [draw, very thick, green, stealth-stealth]

\begin{figure}[htb]
\centering
\begin{tikzpicture}[scale=1.5,auto]
\node (1) at (135:2) [vert-A] {\scriptsize $1$};
\node (i) at (45:2) [vert-B] {\scriptsize $i$};
\node (k) at (-45:2) [vert-A] {\scriptsize $k$};
\node (j) at (-135:2) [vert-B] {\scriptsize $j$};
\node (-1) at (135:1) [vert-A] {\scriptsize $-1$};
\node (-i) at (45:1) [vert-B] {\scriptsize $-i$};
\node (-k) at (-45:1) [vert-A] {\scriptsize $-k$};
\node (-j) at (-135:1) [vert-B] {\scriptsize $-j$};

\path[b] (1) to (i);
\path[b] (i) to (-1);
\path[b] (-1) to (-i);
\path[b] (-i) to (1);

\path[b] (-j) to (k);
\path[b] (-k) to (-j);
\path[b] (j) to (-k);
\path[b] (k) to (j);

\path[r] (-k) to (i);
\path[r] (-i) to (-k);
\path[r] (k) to (-i);
\path[r] (i) to (k);

\path[r] (1) to (j);
\path[r] (j) to (-1);
\path[r] (-1) to (-j);
\path[r] (-j) to (1);

\end{tikzpicture}
\caption{Cayley Graph for $Q_8$ with generating set $\{i, j\}$.}\label{fig:Q8}
\end{figure}
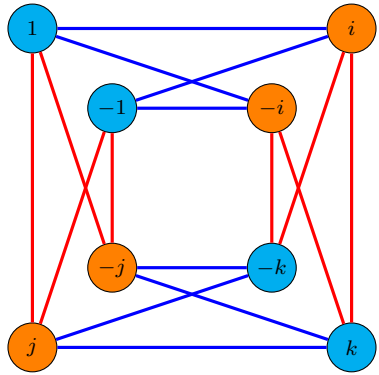

In general, if a group $G$ contains an index 2 subgroup $H$ and we let $S = G - H$, then $\Gamma(G,S)$ is complete bipartite. In this case, we have: 

\begin{theorem} \label{thm:complete-bipartite}
If $G$ is a finite group of order $2n$, with  $n\geq 2$ and $S$ is a generating set such that $\Gamma(G,S)$ is complete bipartite, then Player 2 wins $\REL(G,S)$ and Player 1 wins $\RAV(G,S)$.
\end{theorem}

\begin{proof}
The proof that Player 2 wins $\REL(G,S)$ follows the same argument as in \cref{ex: quaternions}. 

For $\RAV(G,S)$, the game is one of exhaustion. If $|G| = 2n$, then Player 1 has $n$ possible vertices to move to on their first turn, while Player 2 has $n-1$ options due to the game starting at the identity. In general, Player 1 has $n-k$ vertex options after their $k$th turn, while Player 2 has $n-k-1$ vertex options after their $k$th turn. These options always exist because $\Gamma$ is complete bipartite. Hence, Player 2 will exhaust their options before Player 1, and thus Player 1 wins $\RAV(G,S)$.\\
\end{proof}

For any non-trivial finite group $G$, if we let $S = G - \{e\}$, then $\Gamma(G,S)$ is a complete graph. Such a case is also easy to analyze.

\begin{theorem} \label{thm:complete-graph}
If $G$ is a finite group of order at least $3$ and $S$ is a generating set such that $\Gamma(G,S)$ is a complete graph, then Player 1 wins $\REL(G,S)$. Player 1 wins $\RAV(G,S)$ if $|G|$ is even and Player 2 wins $\RAV(G,S)$ if $|G|$ is odd.
\end{theorem}

\begin{proof}
If $\Gamma(G,S)$ is complete and $|G| \geq 3$, then Player 1 wins $\REL(G,S)$ on their second turn by moving back to $e$ since Player 2 may not backtrack to $e$ on their first turn. 
$\RAV(G,S)$ is a game of exhaustion as in the complete bipartite case. If $|G|$ is even, then Player 1 will complete a Hamiltonian path in $\Gamma(G,S)$ on turn $|G|-1$ and thus win $\RAV(G,S)$ since Player 2 will have no available moves on the next turn. If $|G|$ is odd, then Player 2 wins by completing a Hamiltonian path for the same reason.
\end{proof}

While generating sets that yield complete or complete bipartite Cayley graphs allow for quick analysis of $\REL$ and $\RAV$, they are rarely canonical generating sets for groups. In this sense, $Q_8$ is an outlier with its canonical generating set yielding a complete bipartite Cayley graph.

Now suppose two groups $G$ and $H$ have isomorphic, undirected Cayley graphs, $\Gamma(G,S)$ and $\Gamma(H,T)$. 
A natural question is to ask if the games of $\REL$ and $\RAV$ will be the same for both groups. This is indeed the case. If a winning strategy dictates a player move along the edge from $g$ to $gs$ in $\Gamma(G,S)$, then the same player has a winning strategy on the other group by moving along the corresponding edge in $\Gamma(H,T)$. We state this explicitly as the following theorem.

\begin{theorem}\label{thm: iso cayley graphs}
Suppose $\Gamma(G,S)$ and $\Gamma(H,T)$ are isomorphic as undirected Cayley graphs.
A player has a winning strategy for $\REL(G,S)$ (respectively, $\RAV(G,S)$), if and only if that player has a winning strategy for $\REL(H,T)$ (respectively, $\RAV(H,T)$).
\end{theorem}

We provide an example of this result in \cref{ex: dihedral with coxeter gen set is cyclic graph} at the beginning of the next section.

\section{Dihedral Groups}
\label{sec: dihedral groups}

For the dihedral groups $D_n$ of order $2n$, with $n \geq 3$, there are two common generating sets: one is the Coxeter generating set comprised of two reflections; the other is comprised of one reflection and one rotation. First we examine the Coxeter generating set. 

\begin{example}\label{ex: dihedral with coxeter gen set is cyclic graph}
Suppose $S = \{s, t\}$ is a Coxeter generating set for the dihedral group $D_n$. That is,
\[
    D_n = \langle s, t ~|~ s^2=t^2= (st)^n =e \rangle.
\]
In this case, the games $\REL(D_n, S)$ and $\RAV(D_n, S)$ have the same outcomes as $\REL(\Z_{2n}, \{1\})$ and $\RAV(\Z_{2n}, \{1\})$ since the undirected Cayley graphs $\Gamma(\Z_{2n}, \{1\})$ and $\Gamma(D_n, \{s,t\})$ are isomorphic (cf. \cref{thm: iso cayley graphs}). 
\end{example}

Hence, we focus our attention for the rest of this section on the following presentation for the dihedral groups:
\[
    D_n = \langle r, s ~|~ r^n=s^2=rsrs=e \rangle.
\]

\subsection{$\REL(D_n, \{r,s\})$}

In this section, we investigate the Relator Achievement Game on $D_n$ with generating set $\{r,s\}$. Note that each element of $D_n$ can be written uniquely as $r^is^j$, for some integers $i$ and $j$ with $0 \leq i \leq n-1$ and $0 \leq j \leq 1$.

\begin{theorem}\label{thm: 2 player dihedral REL}
If $n$ is odd, then Player 1 has a winning strategy for $\REL(D_n, \{r,s\})$. If $n$ is even, Player 1 has a winning strategy if $n\equiv 2 \mod 6$ while Player 2 has a winning strategy otherwise.
\end{theorem}
Before we begin the proof, we provide some remarks and a lemma that will aid in the proof.
\begin{remark}\label{rem: squares}
First, the Cayley graph $\Gamma(D_n, \{r,s\})$ contains $n$ ``squares", each corresponding to the relation $rsrs = e$. Given the normal form $r^is^j$, where $0\leq i\leq n-1$ and $0\leq j\leq 1$, we number the squares in increasing order by $i$. Square $1$ contains $\{e,s,r,rs\}$, Square $n$ contains $\{r^{n-1}, r^{n-1}s, e, s\}$, and, in general, Square $i$ contains $\{r^{i-1}, r^{i-1}s, r^i, r^is\}$. See \cref{fig:square} and \cref{fig: D5 Cayley Graph}.
\end{remark}

\tikzstyle{vert} = [circle, draw, inner sep=0pt, minimum size=6.5mm]
\tikzstyle{b} = [draw,very thick,blue]
\tikzstyle{bl} = [draw,very thick,black]
\tikzstyle{r} = [draw, very thick, red]
\begin{figure}[ht]
	\centering
	\begin{tikzpicture}[scale=0.75]
	\node (gs) at (45:-2) [vert] {\scriptsize $r^{i-1}s$};
	\node (grs) at (-45:2) [vert] {\scriptsize $r^{i}s$};
	\node (gr) at (45:2) [vert] {\scriptsize $r^{i}$};
	\node (g) at (-45:-2) [vert] {\scriptsize $r^{i-1}$};
	\path[bl] (g) to (gr);
	\path[bl] (g) to (gs);
	\path[bl] (gr) to (grs);
	\path[bl] (grs) to (gs);
	\end{tikzpicture}
	\caption{Square $i$ in $D_n$}
	\label{fig:square}
\end{figure}
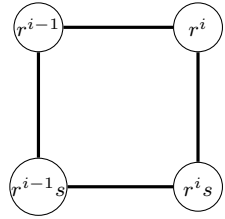 
\tikzstyle{vert} = [circle, draw, inner sep=0pt, minimum size=4.5mm]
\begin{figure}[ht]
	\centering
	\begin{tikzpicture}[scale=0.75]
	\node (e) at (0,2.8) [vert] {\scriptsize $e$};
	\node (r) at (2.7,0.87) [vert] {\scriptsize $r$};
	\node (r2) at (1.7, -2.3) [vert] {\scriptsize $r^2$};
	\node (r3) at (-1.7, -2.3) [vert] {\scriptsize $r^3$};
	\node (r4) at (-2.7, 0.87) [vert] {\scriptsize $r^4$};
	\path[bl] (e) to (r);
	\path[bl] (r) to (r2);
	\path[bl] (r2) to (r3);
	\path[bl] (r3) to (r4);
	\path[bl] (r4) to (e);
	
	\node (s) at (0,1.4) [vert] {\scriptsize $s$};
	\node (rs) at (1.35,0.44) [vert] {\scriptsize $rs$};
	\node (r2s) at (0.83,-1.14) [vert] {\scriptsize $r^2s$};
	\node (r3s) at (-0.83,-1.14) [vert] {\scriptsize $r^3s$};
	\node (r4s) at (-1.35,0.44) [vert] {\scriptsize $r^4s$};
	\path[bl] (s) to (r4s);
	\path[bl] (rs) to (s);
	\path[bl] (r2s) to (rs);
	\path[bl] (r3s) to (r2s);
	\path[bl] (r4s) to (r3s);
	
	\path[bl] (s) to (e);
	\path[bl] (rs) to (r);
	\path[bl] (r2s) to (r2);
	\path[bl] (r3s) to (r3);
	\path[bl] (r4s) to (r4);
	
	\node (1) at (1.1, 1.4) {1};
	\node (2) at (1.55, -0.45) {2};
	\node (3) at (0, -1.55) {3};
	\node (4) at (-1.55, -0.45) {4};
	\node (5) at (-1.1, 1.4) {5};
	\end{tikzpicture}
	\caption{$\Gamma(D_5, \{r,s\})$ with squares 1 through 5}
	\label{fig: D5 Cayley Graph}
\end{figure}
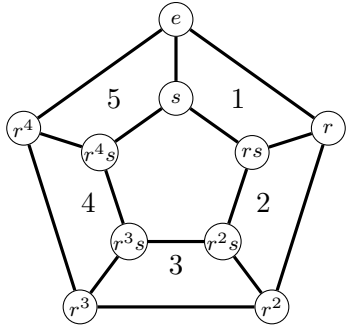 

\begin{remark}\label{rem: no third edge}
If two edges of a square have already been traversed, then neither player will move along a third edge of that square unless it is a winning play since traversing a third edge sets up the opposing player to win on their next turn.
\end{remark}

Because of the previous two remarks, once the first $r$ or $r^{-1}$ edge is chosen, the players will move in one direction, clockwise or counter-clockwise, along the Cayley graph until a cycle is completed. 

\begin{remark}\label{rem: $s$ guarantees next 2 moves}
If a player chooses $s$, then the next two moves (if they exist) are both determined by \Cref{rem: no third edge} and therefore must either both be $r$ or both be $r^{-1}$.
\end{remark}

We now introduce a definition that will be useful in the proof.
\begin{definition}\label{def: entering the square}
We say that a player \emph{enters} Square $i$ at vertex $g$ on turn $k$ if their choice of $s_k \in S \cup S^{-1}$ yields $w_k \equiv_G g$, and none of the edges of Square $i$ have been traversed on any turn $\ell < k$.
\end{definition}
When the context is clear, we will state that a player has entered a square without referring to the specific turn.

The following lemma will be used in the proof of \cref{thm: 2 player dihedral REL}. Note that the appearance of the condition $n \equiv 2 \mod 6$ in \cref{thm: 2 player dihedral REL} is due to this lemma.

\begin{lemma}\label{lem: 3 forward 1 over}
In a game of $\REL(D_n, \{r,s\})$, suppose all moves have occurred on squares 1 through $k-3$, where $5 \leq k \leq n$.
If a player enters square $k-3$ at the vertex $r^{k-4}s$, then that player can guarantee entering square $k$ at vertex $r^{k-1}$.
Similarly, if a player enters square $k-3$ at vertex $r^{k-4}$, then that player can guarantee entering square $k$ at vertex $r^{k-1}s$.
\end{lemma}

\begin{proof}
    We assume that all prior moves have occurred on squares 1 through $k-3$. Without loss of generality, suppose that Player 1 enters square $k-3$ at the vertex $r^{k-4}s$, which must be done via a choice of $r^{-1}$. We then have two cases since Player 2 may either play $r^{-1}$ to move to $r^{k-3}s$ or $s$ to move to $r^{k-4}$.
    
    If Player 2 chooses $r^{-1}$, then Player 1 will follow by choosing $s$ to move to $r^{k-3}$. The next two moves are then forced by \cref{rem: $s$ guarantees next 2 moves} if both players are to avoid making the third edge on a square. Hence Player 2 will move to $r^{k-2}$, and Player 1 will move to $r^{k-1}$, entering square $k$ at this vertex.
    
    If Player 2 chooses $s$, then the next two moves are forced, so Player 1 moves to $r^{k-3}$, and Player 2 moves to $r^{k-2}$. Player 1 then has the option to choose $r$ and enter square $k$ at $r^{k-1}$.
    
    The case where Player 1 enters square $k-3$ at vertex $r^{k-4}$ is similar.
    
\end{proof}

\begin{proof}[Proof of  \Cref{thm: 2 player dihedral REL}]
We consider the cases where $n$ is odd and $n$ is even separately. First suppose $n$ is odd. Then Player 1 has a winning strategy described as follows using the normal form of elements of $D_n$. Player 1 chooses $r$ on their first turn. On all subsequent turns, if Player 1 is given a word equivalent to $r^k$ for some $1\leq k\leq n-1$, then Player 1 chooses $r$ to move to $r^{k+1}$. If Player 1 is given a word equivalent to $r^ks$ for some $1\leq k\leq n-1$, then Player 1 chooses $r^{-1}$ to move to $r^{k+1}s$. See \cref{fig: REL dihedral odd case example} for an example of this strategy.

Suppose Player 1 opens with $r$. By the parity of $n$, if Player 2 only chooses $r$ each turn, then Player 1 will win the game by moving from $r^{n-1}$ to $e$. This means Player 1 lands on vertices equivalent to $r^k$ with $k$ odd, while Player 2 lands on vertices $r^k$ with $k$ even. Hence, at some point, Player 2 must play an $s$ to avoid a loss. When this occurs for the first time, Player 2 moves from $r^k$ to $r^ks$, where $k$ is odd and $1\leq k \leq n-2$.

Now, by \cref{rem: $s$ guarantees next 2 moves}, the next two moves are forced: that is, Player 1 must move from $r^ks$ to $r^{k+1}s$ and then Player 2 must move from $r^{k+1}s$ to $r^{k+2}s$, both playing an $r^{-1}$ generator. If $k+2 = n$, then Player 1 will play an $s$ and win the game at $e$. Otherwise, Player 1 continues to play $r^{-1}$. Again, by the parity of $n$, if Player 2 only plays $r^{-1}$ generators as well, then eventually Player 2 will land at $r^ns = s$, and Player 1 can win at $e$ by playing an $s$.

Hence Player 2 must eventually play another $s$. But if that occurs, we will have Player 2 moving from a vertex of the form $r^ms$ to $r^m$ where $m$ is \emph{even}. Then again, by \cref{rem: $s$ guarantees next 2 moves}, Player 1 must play an $r$ followed by another $r$ by Player 2, landing at $r^{m+2}$. Note that, if $m=n-1$, then Player 1 will have won the game on their forced move, since they will have moved from $r^m$ to $r^{m+1} = r^n = e$. 
Play will continue this way until Player 2 is forced to move to $r^{n-1}$ or $s$, in which case Player 1 wins the game on the subsequent turn.

\tikzstyle{vert} = [circle, draw, inner sep=0pt, minimum size=2.5mm]
\tikzstyle{dot} = [draw,very thick,black, dotted]
\tikzstyle{bl} = [draw, very thick, black]
\begin{figure}
    \centering
    \begin{tikzpicture}[scale=0.85]
        \node (e) at (0, 4) [vert] {};
        \node at (0, 4.4) {$e$};
        \node (r) at (3.1, 2.5) [vert] {};
        \node at (3.5, 2.5) {$r$};
        \node (r2) at (3.9, -0.9) [vert] {};
        \node at (4.3, -0.9) {$r^2$};
        \node (r3) at (1.7, -3.6) [vert] {};
        \node at (2.1, -4) {$r^3$};
        \node (r4) at (-1.7, -3.6) [vert] {};
        \node at (-2.1, -4) {$r^4$};
        \node (r5) at (-3.9, -0.9) [vert] {};
        \node at (-4.3, -0.9) {$r^5$};
        \node (r6) at (-3.1, 2.5) [vert] {};
        \node at (-3.5, 2.5) {$r^6$};
        
        \path[bl] (e) to (r);
        \path[bl] (r) to (r2);
        \path[bl] (r2) to (r3);
        \path[bl] (r3) to (r4);
        \path[bl] (r4) to (r5);
        \path[bl] (r5) to (r6);
        \path[bl] (r6) to (e);
        
        \node (s) at (0, 2) [vert] {};
        \node at (0, 1.7) {$s$};
        \node (rs) at (1.6, 1.2) [vert] {};
        \node at (1.2, 0.8) {$rs$};
        \node (r2s) at (2, -0.4) [vert] {};
        \node at (1.3, -0.2) {$r^2s$};
        \node (r3s) at (0.9, -1.8) [vert] {};
        \node at (0.5, -1.4) {$r^3s$};
        \node (r4s) at (-0.9, -1.8) [vert] {};
        \node at (-0.5, -1.4) {$r^4s$};
        \node (r5s) at (-2, -0.4) [vert] {};
        \node at (-1.3, -0.1) {$r^5s$};
        \node (r6s) at (-1.6, 1.2) [vert] {};
        \node at (-1.2, 0.8) {$r^6s$};
        
        \path[bl] (s) to (rs);
        \path[bl] (rs) to (r2s);
        \path[bl] (r2s) to (r3s);
        \path[bl] (r3s) to (r4s);
        \path[bl] (r4s) to (r5s);
        \path[bl] (r5s) to (r6s);
        \path[bl] (r6s) to (s);
        
        \path[bl] (e) to (s);
        \path[bl] (r) to (rs);
        \path[bl] (r2) to (r2s);
        \path[bl] (r3) to (r3s);
        \path[bl] (r4) to (r4s);
        \path[bl] (r5) to (r5s);
        \path[bl] (r6) to (r6s);
        
        \path[r, -stealth] (e) to (r);
        \path[b, -stealth] (r) to (r2);
        \path[r, -stealth] (r2) to (r3);
        \path[b, -stealth] (r3) to (r3s);
        \path[r, -stealth] (r3s) to (r4s);
        \path[b, -stealth] (r4s) to (r5s);
        \path[r, -stealth] (r5s) to (r6s);
        
    \end{tikzpicture}
    \caption{Example of Player 1 strategy for $\REL(D_7)$ as described in \cref{thm: 2 player dihedral REL}. Player 1 moves are colored `{\color{red}red}' and Player 2 moves are colored `{\color{blue}blue}'. Player 1 only needs to choose $r$ or $r^{-1}$ generators to win. Note that when Player 2 chooses an $s$, the next two moves are forced by \cref{rem: no third edge}. Regardless of Player 2's next move, Player 1 will win the game}
    \label{fig: REL dihedral odd case example}
\end{figure}

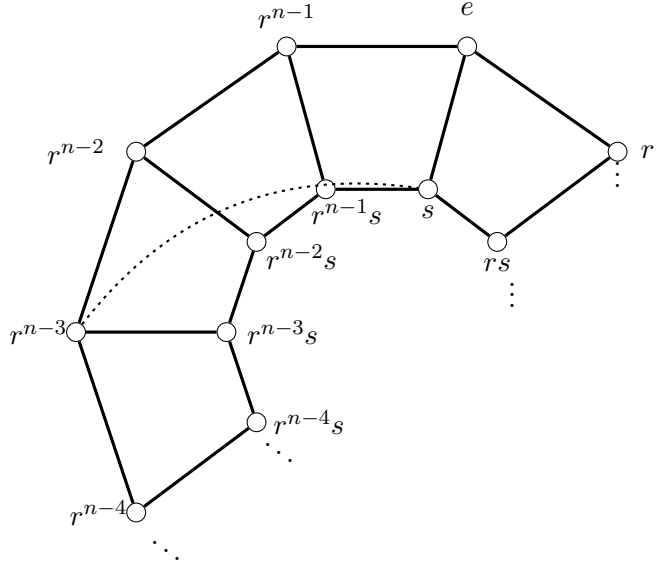
\begin{figure}
    \centering
    \begin{tikzpicture}
        \node at (-2.8, -2.8) {$\ddots$};
        \node (rn4) at (-3.2, -2.4) [vert] {};
        \node at (-3.7,-2.4) {$r^{n-4}$};
        \node (rn3) at (-4,0) [vert] {}; 
        \node at (-4.5,0) {$r^{n-3}$};
        \node (rn2) at (-3.2,2.4) [vert] {};
        \node at (-4,2.4)  {$r^{n-2}$};
        \node (rn1) at (-1.2, 3.8) [vert] {};
        \node at (-1.2,4.2)  {$r^{n-1}$};
        \node (e) at (1.2, 3.8) [vert] {};
        \node at (1.2, 4.3) {$e$};
        \node (r) at (3.2, 2.4) [vert] {};
        \node at (3.6, 2.4) {$r$};
        \node at (3.2, 2.2) {$\vdots$};
        
        \path[bl] (rn4) to (rn3);
        \path[bl] (rn3) to (rn2);
        \path[bl] (rn2) to (rn1);
        \path[bl] (rn1) to (e);
        \path[bl] (e) to (r);
        
        \node at (-1.3, -1.5) {$\ddots$};
        \node (rn4s) at (-1.6, -1.2) [vert] {};
        \node at (-0.9,-1.2) {$r^{n-4}s$};
        \node (rn3s) at (-2,0) [vert] {}; 
        \node at (-1.25,0) {$r^{n-3}s$};
        \node (rn2s) at (-1.6, 1.2) [vert] {};
        \node at (-1, 1)  {$r^{n-2}s$};
        \node (rn1s) at (-0.68, 1.9) [vert] {};
        \node at (-0.4,1.6)  {$r^{n-1}s$};
        \node (s) at (0.68, 1.9) [vert] {};
        \node at (0.68, 1.6) {$s$};
        \node (rs) at (1.6, 1.2) [vert] {};
        \node at (1.6, 0.9) {$rs$};
        \node at (1.8, 0.6) {$\vdots$};
        
        \path[bl] (rn3s) to (rn4s);
        \path[bl] (rn2s) to (rn3s);
        \path[bl] (rn1s) to (rn2s);
        \path[bl] (s) to (rn1s);
        \path[bl] (rs) to (s);
        
        \path[bl] (rs) to (r);
        \path[bl] (s) to (e);
	    \path[bl] (rn1s) to (rn1);
	    \path[bl] (rn2s) to (rn2);
	    \path[bl] (rn3s) to (rn3);
	    \path[bl] (rn4s) to (rn4);
	    
	    \path[dot, thick] (rn3) edge[bend left] (s);
        
    \end{tikzpicture}
    \caption{Portion of a general Cayley graph for $D_n$ with moves as in \cref{thm: 2 player dihedral REL}, the even case with $n \equiv 2 \mod 6$. By \cref{lem: 3 forward 1 over}, if Player 1 enters square $n-2$ at $r^{n-3}$, then they can guarantee reaching vertex $s$. }
    \label{fig: REL dihedral even case figure}
\end{figure}

Now suppose $n$ is even.
If Player 1 chooses $r$, then Player 2 wins by the same strategy used by Player 1 in the odd case. That is, starting at $r^k$, Player 2 will move to $r^{k+1}$; starting at $r^ks$, Player 2 will move to $r^{k+1}s$. Hence, suppose Player 1 begins by choosing $s$. We now consider three cases depending on the value of $n \mod 6$.

First let $n\equiv 2 \mod 6$, so $n = 6k+2$ for some $k$. Then Player 1 wins by entering square $n-2$ at $r^{n-3}$. By \cref{rem: $s$ guarantees next 2 moves}, this guarantees Player 1 enters square 3 at vertex $r^2s$. Now, by repeated use of \cref{lem: 3 forward 1 over}, Player 1 can guarantee entering square $\ell +1$ at $r^{\ell}s$, with $\ell \equiv 2 \mod 6$ and $1 \leq \ell \leq n-2$; and Player 1 guarantees entering square $m+1$ at $r^m$, with $m \equiv 5 \mod 6$ and $1 \leq m \leq n-2$. Since $n-3 = 6k-1 \equiv 5 \mod 6$, Player 1 guarantees entering square $n-2$ at $r^{n-3}$. From there, the same argument as in the proof of \cref{lem: 3 forward 1 over} shows that Player 1 wins at $s$ if Player 2 avoids making the third edge of a square.

If $n\equiv 0 \mod 6$, so $n= 6k$ for some $k$, and Player 1 starts by choosing $s$, then Player 2 wins by entering square $n-2$ at $r^{n-3}s$. Working backwards, this is guaranteed by entering square 4 at $r^3s$. Since Player 1 begins the game with $s$, the next two moves are forced to be $rs$ and $r^2s$. Player 2 then has the choice of moving to $r^3s$. 
By repeated use of \cref{lem: 3 forward 1 over}, Player 2 guarantees entering square $\ell + 1$ at $r^{\ell}s$, with $\ell \equiv 3 \mod 6$, and guarantees entering square $m+1$ at $r^m$, with $m \equiv 0 \mod 6$ (still with $1 \leq \ell, m \leq n-2)$. Since $n-3 = 6k-3 \equiv 3 \mod 6$, Player 2 guarantees entering square $n-2$ at $r^{n-3}s$. From there, the same argument as in the proof of \cref{lem: 3 forward 1 over} shows that Player 2 wins at $e$ if Player 1 avoids making the third edge of a square.

If $n\equiv 4\mod 6$, so $n = 6k+4$ for some $k\geq 1$, and Player 1 starts by choosing $s$, then Player 2 wins again by entering square $n-2$ at $r^{n-3}s$. Working backwards according to \cref{lem: 3 forward 1 over}, this is guaranteed by entering square 2 at $rs$, which is possible given that Player 1 starts by choosing $s$. By repeated use of \cref{lem: 3 forward 1 over}, Player 2 enters square $\ell + 1$ at $r^{\ell}s$, with $\ell \equiv 1 \mod 6$, and guarantee entering square $m+1$ at $r^m$, with $m \equiv 4 \mod 6$. Since $n-3 = 6k+1 \equiv 1 \mod 6$, Player 1 guarantees entering square $n-2$ at $r^{n-3}s$. From there, the same argument as in the proof of \cref{lem: 3 forward 1 over} shows that Player 2 wins at $e$ if Player 1 avoids making the third edge of a square.

Note that if $k = 0$, then \cref{lem: 3 forward 1 over} does not apply since $n=4$. In this case, we note that Player 2 will enter Square 2 at $rs$ and then wins at $e$ if Player 1 avoids making the third edge of a square by the same argument as in the proof of \cref{lem: 3 forward 1 over}.
\end{proof}

\subsection{$\RAV(D_n, \{r,s\})$}
In contrast with the achievement game (\cref{thm: 2 player dihedral REL}), we have that Player 1 has a winning strategy for $\RAV$ for any $n \geq 3$.
This strategy involves the formation of a Hamiltonian path in the Cayley graph (see \cref{fig: horseshoe strategy}).

\begin{theorem} \label{thm: 2-player rav dihedral}
Player 1 has a winning strategy for $\RAV(D_n, \{r,s\})$ for any $n \geq 3$.
\end{theorem}
\begin{proof}

Player 1 has the following strategy: always choose the generator $s$. Indeed, after choosing $s$ for the first turn of the game, the choices for Player 2 are symmetric. Without loss of generality, we assume Player 2 chooses $r^{-1}$. Then by Player 1's strategy, the game word becomes $sr^{-1}s \equiv_{D_n} r$. As Player 2 would lose by choosing $r^{-1}$, they are forced to choose $r$ and the game word is now equivalent to $r^2$.
The game then proceeds in an inductive fashion: for any $2 \leq k \leq n-1$, Player 1's strategy will move them sequentially from the vertex $r^k$ to $r^ks$ for $k$ even and from $r^{k}s$ to $r^{k}$ for $k$ odd. Player 2 is always forced to move from $r^ks$ to $r^{k+1}s$ when $k$ is even and from $r^{k}$ to $r^{k+1}$ when $k$ is odd.

Hence, unless Player 2 chooses to lose the game immediately, the players will proceed until Player 1 arrives at $r^{n-1}$ if $n$ is even or $r^{n-1}s$ if $n$ is odd, completing a Hamiltonian path, and Player 2 loses the game on their next move.
\end{proof}

\tikzstyle{b} = [draw,very thick,blue]
\tikzstyle{r} = [draw, very thick, red]
\tikzstyle{black-r} = [draw,very thick,black]
\tikzstyle{black-s} = [draw, very thick, black]
\tikzstyle{player-1} = [draw, very thick, red, stealth-]
\tikzstyle{player-2} = [draw, very thick, blue, -stealth]

\tikzstyle{vert} = [circle, draw, inner sep=0pt, minimum size=2.5mm]
\tikzstyle{dot} = [draw,very thick,black, dotted]
\begin{figure}[hbt]
    \centering
    \begin{tikzpicture}[scale=0.75]
        \node (r6) at (-3.2, -2.4) [vert] {};
        \node at (-3.7,-2.4) {$r^6$};
        \node (r7) at (-4,0) [vert] {}; 
        \node at (-4.5,0) {$r^7$};
        \node (r8) at (-3.2,2.4) [vert] {};
        \node at (-4,2.4)  {$r^8$};
        \node (r9) at (-1.2, 3.8) [vert] {};
        \node at (-1.2,4.2)  {$r^9$};
        \node (e) at (1.2, 3.8) [vert] {};
        \node at (1.2, 4.3) {$e$};
        \node (r) at (3.2, 2.4) [vert] {};
        \node at (3.6, 2.4) {$r$};
        \node (r2) at (4, 0) [vert] {};
        \node at (4.4, 0) {$r^2$};
        \node (r3) at (3.2, -2.4) [vert] {};
        \node at (3.6, -2.5) {$r^3$};
        \node (r4) at (1.2, -3.8) [vert] {};
        \node at (1.5, -4) {$r^4$};
        \node (r5) at (-1.2, -3.8) [vert] {};
        \node at (-1.2, -4.2) {$r^5$};
    
        \path[black-r] (e) to (r);
        \path[player-2] (r) to (r2);
        \path[black-r] (r2) to (r3);
        \path[player-2] (r3) to (r4);
        \path[black-r] (r4) to (r5);
        \path[player-2] (r5) to (r6);
        \path[black-r] (r6) to (r7);
        \path[player-2] (r7) to (r8);
        \path[black-r] (r8) to (r9);
        \path[black-r] (r9) to (e);
        
        \node (r6s) at (-1.6, -1.2) [vert] {};
        \node at (-1.1, -1.1) {$r^6s$};
        \node (r7s) at (-2,0) [vert] {}; 
        \node at (-1.4,0) {$r^7s$};
        \node (r8s) at (-1.6, 1.2) [vert] {};
        \node at (-1, 1)  {$r^8s$};
        \node (r9s) at (-0.68, 1.9) [vert] {};
        \node at (-0.4,1.6)  {$r^9s$};
        
        \node (s) at (0.68, 1.9) [vert] {};
        \node at (0.68, 1.6) {$s$};
        \node (rs) at (1.6, 1.2) [vert] {};
        \node at (1.3, 0.9) {$rs$};

        \node (r2s) at (2, 0) [vert] {};
        \node at (1.5, 0) {$r^2s$};
        \node (r3s) at (1.6, -1.2) [vert] {};
        \node at (1.1, -1) {$r^3s$};
        \node (r4s) at (0.6, -1.9) [vert] {};
        \node at (0.6, -1.5) {$r^4s$};
        \node (r5s) at (-0.6, -1.9) [vert] {};
        \node at (-0.25, -1.4) {$r^5s$};
        
        \path[player-2] (r8s) to (r9s);
        \path[black-r] (r8s) to (r7s);
        \path[player-2] (r6s) to (r7s);
        \path[black-r] (r6s) to (r5s);
        \path[player-2] (r4s) to (r5s);
        \path[black-r] (r4s) to (r3s);
        \path[player-2] (r2s) to (r3s);
        \path[black-r] (r2s) to (rs);
        \path[player-2] (s) to (rs);
        \path[black-r] (s) to (r9s);
        
        \path[player-1] (s) to (e);
	    \path[player-1] (r) to (rs);
	    \path[player-1] (r2s) to (r2);
	    \path[player-1] (r3) to (r3s);
	    \path[player-1] (r4s) to (r4);
	    \path[player-1] (r5) to (r5s);
	    \path[player-1] (r6s) to (r6);
	    \path[player-1] (r7) to (r7s);
	    \path[player-1] (r8s) to (r8);
	    \path[player-1] (r9) to (r9s);

    \end{tikzpicture}
    \caption{Example of Player 1 winning strategy for $\RAV(D_{10})$. Player 1 moves are colored `{\color{red}red}' and Player 2 moves are colored `{\color{blue}blue}'. Player 1's strategy is to always choose the generator $s$}
    \label{fig: horseshoe strategy}
\end{figure}
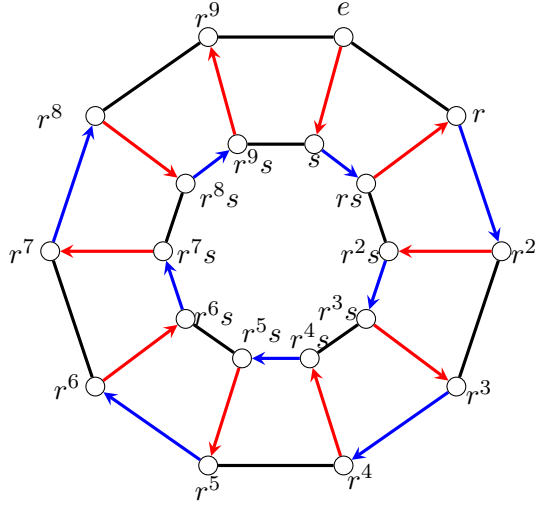

\subsection{$\RAV$ for Groups with an Order Two Generator} 
In \cref{thm: 2-player rav dihedral} we showed that Player 1 had a winning strategy by means of always choosing the order two generator $s$. We can generalize this strategy for $\RAV(G,S)$, where $S$ contains an element of order $2$. This gives another proof of \cref{thm: 2-player rav dihedral}.

\begin{theorem} \label{thm: RAV any group with order two}
Let $G$ be a finite group with generating set $S$ containing an element $s$ of order $2$. Then Player 1 has a winning strategy for the game $\RAV(G,S)$. 
\end{theorem}

\begin{proof}
    The winning strategy of Player 1 is precisely that employed in \cref{thm: 2-player rav dihedral}. That is, Player 1 always chooses the order two generator $s$. Since Player 2 can never choose $s$ due to backtracking, they are forced to choose another element of $S$. 
    
    We first show that a choice of $s$ exists on each turn for Player 1. Indeed, suppose it is Player 1's turn and no such choice is available. Let $v$ denote the vertex in the Cayley graph $\Gamma(G, S)$ representing this point in the game. Because Player 1 has no choice of $s$ available, this means that the edge labeled $s$ from vertex $v$ has been traversed previously. But then the vertex $v$ must have been visited previously, meaning Player 2's last move arriving at $v$ was in fact a losing move for Player 2. Hence, if the choice of generator $s$ is not available for Player 1, then Player 2 already lost the game.
    
    We now show that Player 1's strategy is a winning strategy. Suppose for contradiction that Player 1 choosing $s$ to move from the vertex $v$ to the vertex $w$ is a losing move; that is, this forms the first cycle in the Cayley graph. This means that $w$ has previously been visited. In the case that Player 2 reached $w$ the previous time, then Player 1's strategy implies that they would move to $v$ via choosing $s$. Hence Player 2 actually formed a cycle by moving to $v$ for the second time, a contradiction. 
    
    In the case that Player 1 reached $w$ the previous time, it was from the vertex $v$, so Player 2 again must have formed a cycle by moving to $v$ for the second time.
    \end{proof}
    
As \cref{thm: RAV any group with order two} is quite general, we state a few explicit examples.

\begin{example} \label{ex: Zn and Z2 RAV}
Consider the product $\Z_n \times \Z_2$ with presentation $\langle x, s ~|~ x^n=s^2=xsx^{-1}s=e\rangle$. Then Player 1 has a winning strategy for $\RAV(\Z_n \times \Z_2, \{x,s\})$ by always choosing the generator $s$.
\end{example}

\begin{example} \label{ex: semi-direct and generalized dihedral}
We can generalize \cref{ex: Zn and Z2 RAV} to any product with a cyclic group of order two as follows.
Let $H$ be a finite group with generating set $T$ and let $\{e, s\} = \langle s \rangle \cong \Z_2$ be a cyclic group of order two with generator $s$. Suppose $G = H \rtimes \Z_2$ with canonical generating set
    \[
        S = (T \times \{e\}) \cup (\{e_H\} \times \{s\}).
    \]
Then \cref{thm: RAV any group with order two} implies that Player 1 has a winning strategy for $\RAV(G,S)$ by always choosing the generator $(e_H, s)$.

This applies in particular to the family of \textit{generalized dihedral groups}, which are defined as the the groups $G = H \rtimes \Z_2$ where $H$ is an abelian group and the action of $\Z_2$ on $H$ is that of inversion. 
\end{example}

\begin{remark}
Suppose that $G$ is a group of even order. Then $G$ must contain an element of order two. It follows from \cref{thm: RAV any group with order two} that there exists a generating set $S$ for which Player 1 has a winning strategy for the game of $\RAV(G,S)$.
\end{remark}

\section{$\RAV$ and $\REL$ for Dicyclic Groups}
\label{sec: dicyclic}

\subsection{Dicyclic $\RAV$}

The dicylic group $\dic_n$ of order $4n$ is most commonly written via the following presentation:
\[
    \dic_n = \langle a, x ~|~ a^{2n}=x^4=x^{-1}axa=e \rangle.
\]

From the defining relations, one can show that any $g \in \dic_n$ can be written in a normal form $a^i x^j$, with $0 \leq i < 2n$ and $j \in \{0,1\}$, and with the following relations:
\begin{align*}
    a^ka^{\ell} &= a^{k+\ell} \\
    a^ka^{\ell}x &= a^{k+\ell}x \\
    a^kxa^{\ell} &= a^{k-\ell}x \\
    a^kxa^{\ell}x &= a^{k-\ell+n}.
\end{align*}

\tikzstyle{b} = [draw,very thick,blue]
\tikzstyle{r} = [draw, very thick, red]
\tikzstyle{g} = [draw, very thick, green]
\tikzstyle{vert} = [circle, draw, inner sep=0pt, minimum size=2.5mm]
\tikzstyle{dot} = [draw,very thick,black, dotted]

\begin{figure}[h!bt]
    \centering
    \begin{tikzpicture}[scale=0.75]
        \node (e) at (-1.5, 3.7) [vert] {};
        \node at (-1.5, 4.2) {$e$};
        \node (a) at (1.5,3.7) [vert] {}; 
        \node at (1.5,4.2) {$a$};
        \node (a2) at (3.7,1.5) [vert] {};
        \node at (4.2, 1.5)  {$a^2$};
        \node (a3) at (3.7, -1.5) [vert] {};
        \node at (4.2, -1.5)  {$a^3$};
        \node (a4) at (1.5, -3.7) [vert] {};
        \node at (1.5, -4.2) {$a^4$};
        \node (a5) at (-1.5,-3.7) [vert] {};
        \node at (-1.5,-4.2) {$a^5$};
        \node (a6) at (-3.7, -1.5) [vert] {};
        \node at (-4.2, -1.5) {$a^6$};
        \node (a7) at (-3.7, 1.5) [vert] {};
        \node at (-4.2, 1.5) {$a^7$};
        
        \path[b] (e) to (a);
        \path[b] (a) to (a2);
        \path[b] (a2) to (a3);
        \path[b] (a3) to (a4);
        \path[b] (a4) to (a5);
        \path[b] (a5) to (a6);
        \path[b] (a6) to (a7);
        \path[b] (a7) to (e);

      \node (a6x) at (-0.4, 0.9) [vert] {};
        \node (a7x) at (0.4,0.9) [vert] {}; 
        \node (x) at (0.9,0.4) [vert] {};
        \node (ax) at (0.9, -0.4) [vert] {};
        \node (a2x) at (0.4, -0.9) [vert] {};
        \node (a3x) at (-0.4,-0.9) [vert] {};
        \node (a4x) at (-0.9, -0.4) [vert] {};
        \node (a5x) at (-0.9, 0.4) [vert] {};

        \path[b] (x) to (a7x);
        \path[b] (a7x) to (a6x);
        \path[b] (a6x) to (a5x);
        \path[b] (a5x) to (a4x);
        \path[b] (a4x) to (a3x);
        \path[b] (a3x) to (a2x);
        \path[b] (a2x) to (ax);
        \path[b] (ax) to (x);
        
        \path[r, bend left] (e) to (x);
        \path[r, bend left] (x) to (a4);
        \path[r, bend left] (a4) to (a4x);
        \path[r, bend left] (a4x) to (e);
        
        \path[r, bend left] (a) to (ax);
        \path[r, bend left] (ax) to (a5);
        \path[r, bend left] (a5) to (a5x);
        \path[r, bend left] (a5x) to (a);
        
        \path[r, bend left] (a2) to (a2x);
        \path[r, bend left] (a2x) to (a6);
        \path[r, bend left] (a6) to (a6x);
        \path[r, bend left] (a6x) to (a2);
        
        \path[r, bend left] (a3) to (a3x);
        \path[r, bend left] (a3x) to (a7);
        \path[r, bend left] (a7) to (a7x);
        \path[r, bend left] (a7x) to (a3);
    \end{tikzpicture}
    \caption{Cayley graph for $\dic_4$ with generators $a$ and $x$. The `{\color{blue}blue}' edges correspond to the generator $a$ and the `{\color{red}red}' edges to the generator $x$. On the inner octagon, if one labels the vertices $x$, $ax$, $a^2x$, $\dots$, $a^7x$, in a clockwise order, then a choice of the generator $a$ will move a player counter-clockwise.}
    \label{fig:dic4 a and x}
\end{figure}
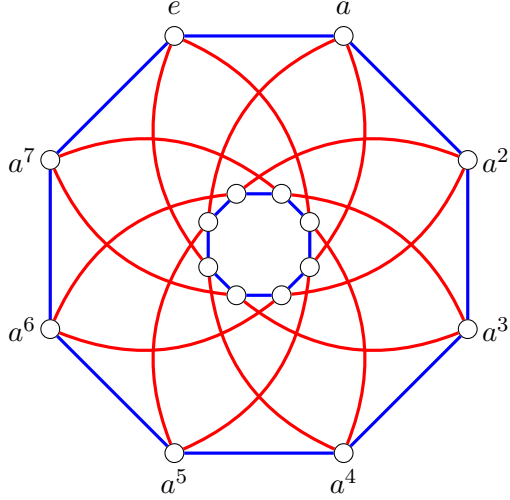

For the game $\RAV(\dic_n, \{a,x\})$, we have the following result, akin to \cref{thm: 2-player rav dihedral} and \cref{thm: RAV any group with order two}. We note here that the generator $x$ is \textit{not} of order two, but it plays a similar role to that of the order two generator from \cref{thm: RAV any group with order two}. Namely, in the normal form for elements of $\dic_n$, the possible powers of $x$ are either $0$ or $1$. So although $x$ has order four, it acts like an element of order two in the normal form.

\begin{theorem} \label{thm: RAV for Dic with two generators}
Player 1 has a winning strategy for $\RAV(\dic_n, \{a,x\})$ for $n \geq 2$.
\end{theorem}

\begin{proof}
    Note that for $n=2$, we have $\dic_2 = Q_8$. Hence, this case is covered by \cref{ex: quaternions}. For the remainder of the proof, suppose $n \geq 3$. See \cref{fig:dic4 a and x} for a Cayley graph of $\dic_4$.
    
    Using the normal form described above, Player 1 has a winning strategy by choosing $x$ on their first turn and then moving from $a^k x$ to $a^k$ by choosing $x^{-1}$ or $a^k$ to $a^k x$ by choosing $x$ on any subsequent turn. Such a move is always available to Player 1, which can be shown via an inductive argument. 
    
    The base case is clear: Player 1 starts with a choice of $x$. Based off Player 2's choice, we have the following three possibilities:
    \[
        xa \equiv_G a^{2n-1}x, \hspace{10pt} x^2\equiv_Ga^n, \hspace{10pt} xa^{-1} \equiv_G ax.
    \]
    For Player 1's next turn, they can thus choose $x^{-1}$, $x$, or $x^{-1}$, respectively, resulting in:
    \[
        xax^{-1} \equiv_G a^{2n-1}, \hspace{10pt} x^3 \equiv_G a^nx, \hspace{10pt} xa^{-1}x^{-1} \equiv_G a.
    \]
    
Note that each of the words $xax^{-1}$, $x^3$, and $xa^{-1}x^{-1}$ are one generator away from a relator. 
Now assume that for the first $m$ turns, Player 1's strategy above has been successfully employed. We want to show that on Player 1's next turn, i.e., on turn $m+2$ of the game, a move from $a^k x$ to $a^k$ via a choice of $x^{-1}$ or a move from $a^k$ to $a^k x$ by choice of $x$ is possible. 

Suppose we are in the first case and $w_{m+1} \equiv_G a^kx$. Assume that Player 2's last turn was via a choice of $x$. This means Player 2's choice of $x$ changed our word from the normal form of $a^k$ to $a^kx$. Since we are assuming Player 1's strategy has been successfully employed for the first $m$ turns, Player 1 already made the move from $a^kx$ to $a^k$ via choosing $x^{-1}$ on the $m$-th turn. Hence, Player 2's choice of $x$ is itself an illegal move.
Similarly, one can show that Player 1's strategy is possible if $w_{m+1} \equiv_G a^k$.
    
    Knowing that Player 1's strategy is always possible, we now show that this strategy is a winning strategy. Assume that Player 1 loses the game.
    If Player 1 forms the first relator at the word equivalent to $z= a^kx$ or $z=a^k$, where $0\leq k\leq 2n-1$, then the word one step prior must have been $y= a^k$ or $y=a^kx$ by Player 1's strategy. 
    Since Player 1 forms a relator at $y$, then one of the two players must have landed at $z$ earlier in the game. By the same argument given in the proof of \cref{thm: RAV any group with order two}, we see that Player 2 necessarily landed at $z$ previously. Hence Player 2 loses the game.
\end{proof}

\begin{figure}[hbt]
    \centering
    \begin{tikzpicture}[scale=0.75]
        \node (e) at (-1.5, 3.7) [vert] {};
        \node at (-1.5, 4.2) {$e$};
        \node (a) at (1.5,3.7) [vert] {}; 
        \node at (1.5,4.2) {$a$};
        \node (a2) at (3.7,1.5) [vert] {};
        \node at (4.2, 1.5)  {$a^2$};
        \node (a3) at (3.7, -1.5) [vert] {};
        \node at (4.2, -1.5)  {$a^3$};
        \node (a4) at (1.5, -3.7) [vert] {};
        \node at (1.5, -4.2) {$a^4$};
        \node (a5) at (-1.5,-3.7) [vert] {};
        \node at (-1.5,-4.2) {$a^5$};
        \node (a6) at (-3.7, -1.5) [vert] {};
        \node at (-4.2, -1.5) {$a^6$};
        \node (a7) at (-3.7, 1.5) [vert] {};
        \node at (-4.2, 1.5) {$a^7$};
        
        \path[b] (e) to (a);
        \path[b] (a) to (a2);
        \path[b] (a2) to (a3);
        \path[b] (a3) to (a4);
        \path[b] (a4) to (a5);
        \path[b] (a5) to (a6);
        \path[b] (a6) to (a7);
        \path[b] (a7) to (e);
        
      \node (a6x) at (-0.4, 0.9) [vert] {};
        \node (a7x) at (0.4,0.9) [vert] {}; 
        \node (x) at (0.9,0.4) [vert] {};
        \node (ax) at (0.9, -0.4) [vert] {};
        \node (a2x) at (0.4, -0.9) [vert] {};
        \node (a3x) at (-0.4,-0.9) [vert] {};
        \node (a4x) at (-0.9, -0.4) [vert] {};
        \node (a5x) at (-0.9, 0.4) [vert] {};

        \path[b] (x) to (a7x);
        \path[b] (a7x) to (a6x);
        \path[b] (a6x) to (a5x);
        \path[b] (a5x) to (a4x);
        \path[b] (a4x) to (a3x);
        \path[b] (a3x) to (a2x);
        \path[b] (a2x) to (ax);
        \path[b] (ax) to (x);
        
        \path[r, bend right] (e) to (a4x);
        \path[r, bend right] (a4x) to (a4);
        \path[r, bend right] (a4) to (x);
        \path[r, bend right] (x) to (e);
        
        \path[r, bend right] (a) to (a5x);
        \path[r, bend right] (a5x) to (a5);
        \path[r, bend right] (a5) to (ax);
        \path[r, bend right] (ax) to (a);
        
        \path[r, bend right] (a2) to (a6x);
        \path[r, bend right] (a6x) to (a6);
        \path[r, bend right] (a6) to (a2x);
        \path[r, bend right] (a2x) to (a2);
    
        \path[r, bend right] (a3) to (a7x);
        \path[r, bend right] (a7x) to (a7);
        \path[r, bend right] (a7) to (a3x);
        \path[r, bend right] (a3x) to (a3);
        
        \path[g] (e) to (ax);
        \path[g] (ax) to (a4);
        \path[g] (a4) to (a5x);
        \path[g] (a5x) to (e);
        
        \path[g] (a) to (a2x);
        \path[g] (a2x) to (a5);
        \path[g] (a5) to (a6x);
        \path[g] (a6x) to (a);

        \path[g] (a2) to (a3x);
        \path[g] (a3x) to (a6);
        \path[g] (a6) to (a7x);
        \path[g] (a7x) to (a2);

        \path[g] (a3) to (a4x);
        \path[g] (a4x) to (a7);
        \path[g] (a7) to (x);
        \path[g] (x) to (a3);
        
    \end{tikzpicture}
    \caption{Cayley graph for $\dic_4$ with generators $a$, $b$ and $c$. The `{\color{blue}blue}' edges correspond to the generator $a$, the `{\color{red}red}' to the generator $b$, and the `{\color{green}green}' edges to $c$. One has a normal form $a^ib^j$ with $0 \leq i < 2n$ and $j \in \{0,1\}$ and the inner vertices are labeled $b$, $ab$, $ab^2$, $\dots$, $ab^7$, in a clockwise manner.} 
    \label{fig:dic4 abc}
\end{figure}
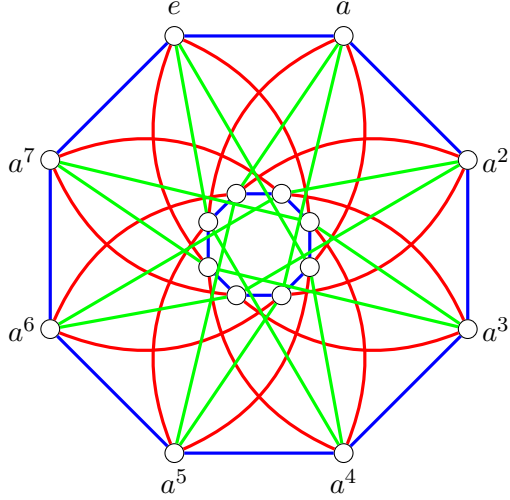

There is another common presentation for the dicyclic groups, namely, as an instance of a triangle group or binary von Dyck group:
    \[
        \dic_n = \langle a, b, c ~|~ a^n=b^2=c^2=abc \rangle.
    \]
See \cref{fig:dic4 abc} for the Cayley graph of $\dic_4$ with this presentation.

Note that the triangle presentation for $\dic_n$ is isomorphic to the one given above via the mapping
 \[
     a \mapsto a, x \mapsto b^{-1}, a x^{-1} \mapsto c.
    \]
For the game of $\RAV(\dic_n, \{a,b,c\})$, we have the same results as \cref{thm: RAV for Dic with two generators}.

\begin{theorem} \label{thm: RAV for dic with three generators}
Player 1 has a winning strategy for $\RAV(\dic_n, \{a,b,c\})$.
\end{theorem}

\begin{proof}
Because we are still dealing with the same group as before, and by the isomorphism described above, we can describe group elements via the normal form $a^i b^j$ with $0 \leq i < 2n$ and $j \in \{0,1\}$. Then Player 1 has a winning strategy by choosing $b$ on their first turn, and then moving from $a^k b$ to $a^k$ by choosing $b^{-1}$, or moving from $a^k$ to $a^k b$ by choosing $b$ on their subsequent turns. The only addition to the previous argument is accounting for the generator $c$. Because $c=ab$, we have $bc = bab \equiv_{\dic_n} a^{n-1}$ and $b c^{-1} = b(b^{-1}a^{-1}) \equiv_{\dic_n} a^{2n-1}$. Hence, the base case given in \cref{thm: RAV for Dic with two generators} can be extended to include the choice of $c$ or $c^{-1}$. The rest of the proof follows exactly as in \cref{thm: RAV for Dic with two generators}.
\end{proof}

\subsection{Dicyclic $\REL$}

We prove in this subsection winning strategies for the Relator Achievement Game for both canonical generating sets of $\dic_n$.

In \cref{fig: for the long proof} we give a simplified, partial Cayley graph for $\dic_n$, labeled with respect to the generating set $\{a,x\}$, that may provide a visual aid for the proof of \cref{thm: REL for odd dicyclic} below. In \cref{fig: for the long proof}, the inner and outer $2n$-gons are given by concentric circles. Instead of drawing all edges, note that a choice of the generator $a$ moves one clockwise on the outer circle, but counter-clockwise on the inner circle. A choice of the generator $x$ or $x^{-1}$ will move one from the inner to outer circle or vice versa.

\begin{theorem} \label{thm: REL for odd dicyclic}
Player 1 has a winning strategy for $\REL(\dic_n, \{a,x\})$ for odd $n \geq 3$.
\end{theorem}

\begin{proof}

Player 1 begins by choosing\footnote{By a symmetrical argument, Player 1 can choose either $a$ or $a^{-1}$ for their first move. For simplicity, we choose $a$} $a$ and will continue to choose $a$ until Player 2 chooses $x$ or $x^{-1}$. We first show that Player 2 loses if they always choose $a$ and in fact that Player 2 must choose $x$ or $x^{-1}$ before reaching $a^{n-2}$.
If both players choose only $a$, Player 1 will arrive at $a^{n-2}$. Note that Player 1 lands at words equivalent to $a^r$, $r$ odd, while Player 2 lands at words equivalent to $a^r$, with $r$ even. Player 2 has only three possible moves from $a^{n-2}$, and we can explore each one to see that Player 1 necessarily has a winning strategy from this position:
    \begin{enumerate}
        \item[(1A)] Suppose Player 2 chooses another $a$ to move from $a^{n-2}$ to $a^{n-1}$. Then Player 1 will choose the generator $x$, arriving at $a^{n-1}x$. If Player 2 follows with either another $x$ or an $a^{-1}$, then Player 1 wins at $e$. Otherwise, if Player 2 chooses $a$, then Player 1 wins at the previously visited $a^{n-2}$.
        
        \item[(1B)] Suppose Player 2 chooses $x$, moving from $a^{n-2}$ to $a^{n-2}x$. Then Player 1 will choose $a^{-1}$, arriving at $a^{n-1}x$. Following the argument as in the previous case, Player 1 will win regardless of Player 2's next move. 
        
        \item[(1C)] Lastly, suppose Player 2 chooses $x^{-1}$, taking us from $a^{n-2}$ to $a^{2n-2}x$. Then Player 1 will play $a^{-1}$, taking us to $a^{2n-1}x$. 
        If Player 2 follows with either $a^{-1}$ or $x^{-1}$, then Player 1 wins at $e$. Otherwise, if Player 2 chooses $x$, then Player 1 wins at the previously visited $a^{n-2}$.
    \end{enumerate}

Hence, we see from the above that Player 2 cannot choose $a$ on each turn and win the game (despite naively believing that might be the case, considering the even order of $a$ in $G$).

From here we know that Player 2 must choose $x$ or $x^{-1}$ before hitting this fulcrum point of $a^{n-2}$. 
Suppose $a^{\ell}$ is the position at which Player 2 chooses to play $x$ or $x^{-1}$, with $1 \leq \ell \leq n-4$ (see \cref{fig: for the long proof} for a picture). Note that $\ell$ is necessarily odd. We examine two cases. 

\begin{enumerate}
    \item[(2A)] 
     
    Suppose Player 2 chooses $x^{-1}$, taking us to $a^{n+\ell}x$. Player 1's strategy is to play $a^{-1}$ until Player 2 chooses a $x$ or $x^{-1}$ again, or until Player 2 moves to $x$, resulting in a Player 1 win at $e$. Note that, since $\ell$ is odd, $n+\ell$ is even, so Player 2 will be the player to arrive at $x$ should both players only choose $a^{-1}$.

    \item[(2B)] Now suppose Player 2 chooses $x$, taking us to $a^{\ell}x$. Player 1 then chooses $a^{-1}$. 
    We now show that Player 2 must choose a $x$ or $x^{-1}$ before arriving at the fulcrum point $a^{n-1}x$. Indeed, suppose both players play $a^{-1}$ until Player 2 moves to $a^{n-2}x$. Then Player 1 chooses $a^{-1}$, arriving at $a^{n-1}x$.

    From this point, Player 2 loses if they choose $a^{-1}$ to move to $a^nx$ since Player 1 can move to $e$. Player 2 also loses at $e$ if they choose $x$ to move to $a^{2n-1}$. Thus Player 2 is forced to choose $x^{-1}$ to move to $a^{n-1}$. Player 1 follows by choosing $a$ to move to $a^n$. From here Player 2 is forced to choose $a$ as well to move to $a^{n+1}$ since $x$ and $x^{-1}$ both allow Player 1 to win at $e$ the following turn.
    
    We now observe that Player 2 loses if both players continue to play $a$ since Player 2 would arrive at $a^{n+\ell}$ where Player 1 has a winning strategy by choosing $x^{-1}$ to move to $a^{\ell}$. Therefore Player 2 must choose $x$ or $x^{-1}$ prior to hitting $a^{n+\ell}$. That is, Player 2 will choose $x$ or $x^{-1}$ from $a^{n+k}$ where $2\leq k\leq \ell-1$. If Player 2 chooses $x^{-1}$, then they move from $a^{n+k}$ to $a^kx$. From there Player 1 chooses $x^{-1}$ to win at $a^k$ by our assumption on $k$. If Player 2 chooses $x$, they arrive at $a^{n+k}x$, and Player 1 chooses $x$ to win at $a^k$.

\end{enumerate}
  In either case, Player 2 must do a $x$ or $x^{-1}$ again before reaching the respective fulcra $x$ and $a^{n-1}x$.

We now examine Player 2 playing a $x$ or $x^{-1}$ from the cases (2A) and (2B) above and notice that, except for certain extremal cases, we return to a previous case. We show that in the extremal cases, Player 1 has a winning strategy. 

\begin{enumerate}
    \item[(3A)] To avoid reaching the fulcrum points $a^{n-2}$ and $a^{n-1}x$, from which Player 1 has winning strategies as described above, suppose Player 2 chooses $x$ to move from $a^{\ell}$ to $a^{\ell}x$ and then later chooses $x$ to move from $a^kx$ to $a^{n+k}$ where $\ell+1\leq k\leq n-3$ and $k$ is even. Player 1 will choose to play $a$. Since $n+k$ is odd, if both players play $a$, then Player 2 will arrive at $a^{2n-1}$, losing the game at $e$. Thus Player 2 must play $x$ or $x^{-1}$ before this point. That is, Player 2 will choose $x$ or $x^{-1}$ from $a^{n+m}$ where $m$ is odd and $k+1\leq m\leq n-2$.
    
    If Player 2 plays a $x$, they move to $a^{n+m}x$, which places us in the situation of case of (2A), but with $m>\ell$.
    
    If Player 2 plays a $x^{-1}$, they move to $a^mx$. This is the situation of the case (2B) above, but with $m>\ell$.
    
    \item[(3B)] To avoid reaching the fulcrum points $a^{n-2}$ and $a^{n-1}x$, suppose instead that Player 2 chooses $x^{-1}$ from $a^kx$ to arrive at $a^k$ where $\ell+1\leq k\leq n-3$ and $k$ is even. If $k= \ell + 1$, then Player 1 wins at $a^{\ell}$. Otherwise, since $k$ is even, we are back in the same situation as before cases (2A) and (2B) - namely, Player 2 must choose a $x$ or $x^{-1}$ before arriving at the word $a^{n-2}$. Note that if $k=n-3$, then Player 1 has a winning strategy as described in (1A), (1B), and (1C).

    \item[(3C)] To avoid reaching the fulcrum points $a^{n-2}$ and $x$, suppose Player 2 moves from $a^{\ell}$ to $a^{\ell}x^{-1} = a^{n+\ell}x$, where $\ell$ is odd, and then later
    Player 2 chooses $x$ from $a^{n+k}x$ to $a^k$ for $\ell+1\leq k\leq n-1$, where $k$ is even. 
    
    If $k<n-1$, then we are in the same situation as before cases (2A) and (2B); that is, where Player 2 loses if both players only choose $a$ and get to $a^{n-2}$. Note Player 1 will win in the case of $k=\ell + 1$ by playing an $a^{-1}$ their next turn.
    
    If $k=n-1$, then this is not subsumed by the above case. In this case, Player 1 moves to $a^n$ by choosing $a$. If both players continue playing $a$, then Player 2 loses when they arrive at $a^{n+\ell}$ since Player 1 can choose $x^{-1}$ to move to $a^{\ell}x$. Thus Player 2 must choose $x$ or $x^{-1}$ from $a^{n+m}$ for some $0\leq m\leq \ell-1$, where $m$ is even.
    
    If Player 2 chooses $x$, they arrive at $a^{n+m}x$. Then Player 1 will choose $x$ to move to $a^m$ and win since $m \leq \ell - 1$. If Player 2 chooses $x^{-1}$, they arrive at $a^mx$; then Player 1 will choose $x^{-1}$ to move to $a^m$ and win.
    
    \item[(3D)] To avoid reaching the fulcrum points $a^{n-2}$ and $x$, suppose Player 2 chooses $x^{-1}$ from $a^{\ell}$ to $a^{n+\ell}x$ and then later chooses $x^{-1}$ from $a^{n+k}x$ to arrive at $a^{n+k}$ where $\ell+1\leq k\leq n-1$ and $k$ is even. 
    
    Firstly, for the extremal case of $k=n-1$, Player 1 will win at $e$ by playing $a$. Now suppose $ \ell + 1 \leq k < n-1$. Note that $n+k$ is odd, so Player 2 loses at $e$ if both players play only $a$. Therefore, Player 2 must choose $x$ or $x^{-1}$ from $a^{n+s}$ for some $k+1\leq s\leq n-2$, where $s$ is odd. 
    
    If Player 2 chooses $x$, they arrive at $a^{n+s}x$. This is the same situation as case (2A), 
    but with $s>\ell$. 
    
    If Player 2 chooses $x^{-1}$, they arrive at $a^sx$, and this is the same situation as case (2B), but with $s > \ell$.
\end{enumerate}

As demonstrated in cases (3A-D), Player 1 will necessarily win the game, or cycle back into a previous case. As $\dic_n$ is finite, the game will eventually end, resulting in a Player 1 win.

\end{proof}

\tikzstyle{vert} = [circle, draw, inner sep=0pt, minimum size=2.5mm, fill=white]
\begin{figure}[hbt]
    \centering
    \begin{tikzpicture}[scale=0.85]
    
    \draw (0,0) circle (2);
    \draw (0,0) circle (4);
    
         \node (e) at (0, 4) [vert, fill, black] {};
         \node at (0, 4.4) {$e$};
         \node (x) at (0,2) [vert, fill, black] {};
         \node at (0, 2.4) {$x$};
         \node (an) at (0, -4) [vert] {};
         \node at (0, -4.4) {$a^n$};
         \node (anx) at (0, -2) [vert] {};
         \node at (-0.2, -2.4) {$a^n x$};
         
         \node (al) at (3.4,2.2) [vert] {}; 
         \node at (3.8,2.6) {$a^{\ell}$};
         \node (alx) at (1.7, 1.1) [vert] {};
         \node at (2.1, 1.5)  {$a^{\ell}x$};

        \node (an-1) at (1.77, -3.59) [vert] {};
        \node at (2.17, -3.99) {$a^{n-1}$};
        \node (an-1x) at (0.88, -1.79) [vert, fill, black] {};
        \node at (1.28, -2.2) {$a^{n-1}x$};
        
        \node (an-2) at (2.83, -2.82) [vert, black, fill] {};
        \node at (3.23, -3.22) {$a^{n-2}$};
        \node (an-2x) at (1.42, -1.41) [vert] {};
        \node at (2.2, -1.5) {$a^{n-2}x$};
        
        \node (an+l) at (-3.3, -2.2) [vert] {};
        \node at (-3.9, -2.4) {$a^{n+\ell}$};
        \node (an+lx) at (-1.7, -1.1) [vert] {};
        \node at (-2.1, -1.5) {$a^{n+\ell}x$};
        
        \node (a2n-1x) at (-0.88, 1.79) [vert] {};
        \node at (-1.5, 2.3) {$a^{2n-1}x$};

    \end{tikzpicture}
    \caption{A simplified way to visualize the Cayley graph of $\dic_n$, which is comprised of two $2n$-gons, represented here by concentric circles. The generator $a$ moves one clockwise on the outer circle, but counter-clockwise on the inner circle. The generator $x$ moves one from the inner to outer circle or vice versa. The black-filled vertices represent the fulcrum points as defined in \cref{thm: REL for odd dicyclic}.
    }
    \label{fig: for the long proof}
\end{figure}
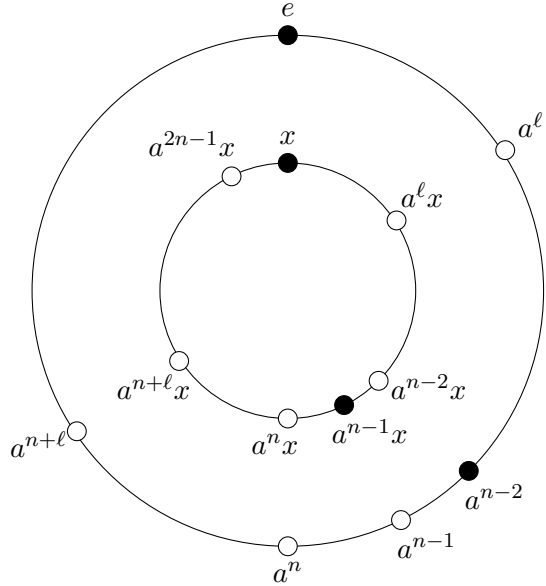

\begin{theorem} \label{thm: REL for dic even}
Player 2 has a winning strategy $\REL(\dic_n, \{a,x\})$ when $n$ is even.
\end{theorem}

\begin{proof}
Player 2 has a winning strategy via mirroring Player 1. That is, if Player 1 chooses a generator $s$ on their turn, Player 2 will follow with $s$ on their turn. Since $x^2 = (x^{-1})^2 = a^n$ and $n$ is even, we note that this strategy implies Player 1 only lands at words equivalent to $a^k$ where $k$ is odd or $a^{\ell}x$ where $\ell$ is even. Meanwhile, Player 2 will only land at words equivalent to $a^k$ where $k$ is even.

To show that Player 2 has a winning strategy, we assume for contradiction that Player 1 has a winning strategy. There are two cases: Player 1 can either win at $a^k$ for $k$ odd or $a^{\ell}x$ for $\ell$ even as stated above. 

First suppose that Player 1 wins at a word equivalent to $a^k$, where $k$ is odd. This means that they arrive at a word equivalent to $a^k$ for the second time in the game.
Since Player 2 only lands at words equivalent to $a^{\ell}$ where $\ell$ is even, Player 1 could only have arrived at a word equivalent to $a^k$ from $a^{k-1}$ or $a^{k+1}$. Without loss of generality, suppose it was from $a^{k-1}$. Then Player 2 moved to $a^{k+1}$ on the subsequent turn. Thus, the second time Player 1 arrives at a word equivalent to $a^k$, it must be from a previously visited word, implying that Player 2 completed a relator on the previous turn, a contradiction.

Now suppose that Player 1 wins at a word equivalent to $a^{\ell}x$, where $\ell$ is even. The first time that Player 1 arrived at $a^{\ell}x$ must have been from $a^{\ell}$ or $a^{n+\ell}$. Player 2 then would have moved to the other. Thus, upon reaching $a^{\ell}x$ for the second time, Player 1 would have again moved from $a^{\ell}$ or $a^{n+\ell}$, both of which would have been visited for a second time. Hence Player 2 won the previous turn, a contradiction.
\end{proof}

\begin{remark} \label{rem: triangle relators}
For $\dic_n$ with generating set $\{a,b,c\}$ there exist the following relators of length three:
\begin{multicols}{4}
\begin{itemize}
    \item[] $abc^{-1}$
\item[] $ba^{-1}c^{-1}$
\item[] $cb^{-1}a^{-1}$
\item[] $a^{-1}cb^{-1}$
\item[] $b^{-1}a^{-1}c$
\item[] $c^{-1}ab$
\item[] $ab^{-1}c$
\item[] $bc^{-1}a$
\item[] $cab^{-1}$
\item[] $a^{-1}c^{-1}b$
\item[] $b^{-1}ca$
\item[] $c^{-1}ba^{-1}$
\end{itemize}
\end{multicols}

The above relators form triangles in the Cayley graph of $\dic_n$. Henceforth, when referencing these relators, we shall call them \textit{triangle relators}.
\end{remark}

Despite the addition of the third generator, we can see that Player 2 has the same winning strategy for $\REL(\dic_n, \{a,b,c\})$ as they did for $\REL(\dic_n, \{a,x\})$ when $n$ is even. 

\begin{theorem}
Player 2 has a winning strategy $\REL(\dic_n, \{a,b,c\})$ for $n$ even.
\end{theorem}

The same argument as described in \cref{thm: REL for dic even} applies here. Recall we can describe every element of $\dic_n$ via the normal form $a^ib^j$ where $0 \leq i < 2n$ and $j \in {0,1}$. From \cref{thm: REL for dic even}, we know that Player 1 can arrive at words equivalent to $a^{\ell}b$ for $\ell$ even and $a^k$ for $k$ odd. However, with the addition of the generator $c$, Player 1 can also arrive at words equivalent to $a^{\ell}b$ for $\ell$ odd. Player 2 still can only arrive at words equivalent to $a^k$ where $k$ is even. We leave the remaining details to the reader.

As opposed to the game $\REL(\dic_n, \{a,x\})$,
Player 2 has a winning strategy for all $n \geq 2$. Note that one can still use \cref{fig: for the long proof} as a visual aid for \cref{thm: rel odd three generators}  by replacing $x$ with $b$ and using that $c=ab$.

\begin{theorem} \label{thm: rel odd three generators}
Player 2 has a winning strategy for $\REL(\dic_n, \{a,b,c\})$ when $n$ is odd.
\end{theorem}

\begin{proof}
We will show first that Player 1 cannot win if they begin the game with a $b$ or $c$ (and via symmetry, neither with $b^{-1}$ nor $c^{-1}$). 

Without loss of generality \footnote{for the case where Player 1 chooses $c$, one simply makes the following changes: $b$ changes to $c$, $c$ to $b$, and $a^{\pm 1}$ to $a^{\mp 1}$.}, suppose Player 1 chooses $b$ on their first turn. Player 2's strategy here is to mirror any play of $b^{\pm 1}$ or $c^{\pm 1}$ with the same generator, since the defining relations give us that $b^2=c^2=a^n$. Hence, after one turn each, the game word is equivalent to $a^n$.
 
 Because of the relator $b^4$, Player 1 cannot play another $b$. Similarly, since $c^2=(c^{-1})^2=a^n$, Player 1 cannot play $c$ or $c^{-1}$, either. Therefore, a choice of either $a$ or $a^{-1}$ is forced for Player 1. However, by the triangle relator $ba^{-1}c^{-1}$, Player 1 cannot choose $a^{-1}$, so they must choose $a$. 
 
 Note that if both players continue to choose $a$ each turn, then Player 1 moves from positions $a^r$ to $a^{r+1}$, where $r$ is odd, and Player 2 the opposite. Hence, Player 1 would eventually win at $e$ unless Player 2 chooses a generator other than $a$. By the triangle relators of \cref{rem: triangle relators}, Player 2 is limited to $c$ or $c^{-1}$. We will show that, should Player 2 wait as long as possible and play $c^{-1}$, then they will win the game, thus forcing Player 1 to act first.
 
 Suppose Player 2 plays $c^{-1}$ at the last available opportunity (assuming up until this point, both players only have chosen $a$). This means Player 2 moves from $a^{2n-2}$ to $a^{n-1}b$. An examination of Player 1's choices of moves shows that Player 1 is forced to mirror Player 2 and choose $c^{-1}$ for their next turn, thus moving from $a^{n-1}b$ to $a^{n-2}$. 
 
 If Player 2 chooses $a^{-1}$, then they will have locked Player 1 into an unwinnable situation. That is, if both players continue choosing $a^{-1}$, then Player 2 would eventually win at a word equivalent to $e$ by parity of $n$. Should Player 1 play a generator different than $a^{-1}$, then Player 2 will mirror that choice, but this will result in a move taking the game to a word equivalent to $a^{n+\ell}$, where $2 \leq \ell \leq n-3$. However, words equivalent to all of these elements have already been visited, so Player 2 would win the game. Therefore, we can indeed conclude that Player 1 must make a move other than $a$ before Player 2 chooses $c^{-1}$ from $a^{2n-2}$.
 
 Now suppose Player 1 plays a $c$ or $c^{-1}$ from a word equivalent to $a^{n+k}$, where $k$ is even and $2\leq k \leq n-3$ (again, by \cref{rem: triangle relators}, Player 1 is limited to choosing from just $c$ or $c^{-1}$). Then Player 2's strategy is again to mirror Player 1. Regardless of Player 1's choice, the game will move to a word equivalent to $a^k$ by this strategy.
 
 We note that whether both players choose $c$ or both players choose $c^{-1}$ on the previous two moves, Player 1 can not choose $a$ by \cref{rem: triangle relators} and cannot choose $b$, $b^{-1}$, $c$, $c^{-1}$ since Player 2 can mirror them, finishing a relator. Thus, Player 1 is forced to choose $a^{-1}$. Note that this moves Player 1 from $a^{k}$ to $a^{k - 1}$, where $k$ is even. Player 2 will choose to play $a^{-1}$ since parity will result in a win at a word equivalent to $e$ should both players continue to choose $a^{-1}$. Thus Player 1 must choose $b$ or $b^{-1}$ from a word equivalent to $a^m$, where $2\leq m\leq k-2$. However, Player 2 will mirror them, landing at a word equivalent to $a^{n+m}$, which has already been visited.
 
Having shown that Player 1 will necessarily lose if they choose $b$ or $c$ on their first turn, we consider the case where Player 1 chooses $a$ first. We first consider if both players repeatedly choose $a$. If this is the case, then Player 2 wins at a word equivalent to $e$ by parity due to the relator $a^{2n}$. Therefore, Player 1 must choose something other than $a$ at some point. If they do so from a word equivalent to $a^{n+k}$, where $1\leq k\leq n-2$, then Player 2 will mirror Player 1 to arrive at $a^k$ and win. Thus Player 1, by \cref{rem: triangle relators}, must choose $c$ or $c^{-1}$ from $a^{\ell}$, where $2\leq \ell\leq n-1$ and $\ell$ is even. If they do so from $a^{n-1}$, this results in a loss since they will arrive at $a^nb$ or $b$, both of which lead to a Player 2 win the following turn. Hence we may assume $2\leq \ell \leq n-3$. 

After Player 1 chooses $c$ or $c^{-1}$, Player 2 will mirror Player 1 and move to a word equivalent to $a^{n+\ell}$. To avoid losing on the following turn, Player 1 is forced to choose $a^{-1}$. We note that Player 1 moves to $a^{r}$, with $r$ even, so Player 1 would win at $a^{\ell}$ if both players continue to choose $a^{-1}$ since $\ell$ is even. Thus Player 2 must choose $b$ or $b^{-1}$ (they cannot choose $c$ or $c^{-1}$ due to \cref{rem: triangle relators}) before Player 1 arrives at $a^{\ell}$. Suppose the game word is equivalent to $a^{\ell +2}$ and Player 2 chooses $b^{-1}$. Then Player 2 will arrive at a word equivalent to $a^{n+\ell+2}b$. By \cref{rem: triangle relators}, Player 1 loses if they choose generators $c, c^{-1},$ or $a^{-1}$. If Player 1 chooses $a$, then they arrive at a word equivalent to $a^{n+\ell+1}b$. But then Player 2 chooses $c^{-1}$ to win at a word equivalent to $a^{n+\ell}$.

Lastly, suppose Player 1 mirrors Player 2 and also chooses $b^{-1}$, arriving at a word equivalent to $a^{n+\ell+2}$. Then Player 2 will choose $a$ afterward. We note that if both players continue to choose $a$, then Player 2 will win at a word equivalent to $e$ due to parity. Thus Player 1 must choose $c$ or $c^{-1}$ from a word equivalent to $a^{n+m}$, where $\ell+3\leq m\leq n-2$ and $m$ is odd. In either case, Player 2 will mirror them and win at a word equivalent to $a^m$. 

Thus, we have shown that Player 1 loses if they play $a^{-1}$ until they arrive at a word equivalent to $a^{\ell+2}$. Hence they must play $b$ or $b^{-1}$ from a word equivalent to $a^s$, where $\ell+3\leq s\leq n+\ell-2$ and $s$ is odd. 
Suppose $s\geq n$. Then Player 2 will mirror Player 1 and arrive at a word equivalent to $a^{n+s}$, which has been previously visited. 
If $s \leq n-1$, then Player 2 will still mirror Player 1 to move to a word equivalent to $a^{n+s}$, which has not yet been visited. Player 1 is then forced to choose $a$ by \cref{rem: triangle relators}. Due to parity, Player 2 will win if both players continue to play $a$, so Player 1 must choose $c$ or $c^{-1}$ from a word equivalent to $a^{n+t}$, where $s+2\leq t\leq n-1$. However, Player 2 will then mirror them and land at a word equivalent to $a^t$, winning the game. 
\end{proof}

\section{$\REL$ for Products of Cyclic Groups}
\label{sec: products of groups}

In this section, we consider the group $\Z_n\times \Z_m$ with the following presentation
$$
\langle a,b \; | \; a^n = b^m = aba^{-1}b^{-1} = e\rangle.
$$

\begin{theorem} \label{theorem: Zn-Zm mod pm1 and 0}
Consider the game $\REL(\Z_n\times \Z_m, \{a,b\})$, where $n\geq m-1$ and $n, m \geq 3$.
\begin{enumerate}
    \item If $n\equiv \pm 1\mod m$, then Player 1 has a winning strategy. 
    \item If $n\equiv 0\mod m$, then Player 2 has a winning strategy.
    \end{enumerate}
\end{theorem}
\begin{proof}
Let $G\cong \Z_n\times \Z_m$. Note that the Cayley graph for $G$ can be visualized as an $n$-gon of $m$-gons (see \cref{fig: for thm 4.1} for a partial Cayley graph example).

First suppose $n\equiv 1\mod m$, so $n = km+1$ for some $k\geq 1$. We show that Player 1 wins using the fact that $a^n b^{n-1} = e$ since $n-1 = km$. The strategy is as follows: Player 1 begins by choosing $a$ and then chooses $a$ whenever Player 2 chooses $b$ and chooses $b$ whenever Player 2 chooses $a$. Note that, when visualizing the Cayley graph for $G$, Player 1's strategy means that no more than two consecutive moves are made on any $m$-gon. We show that no commutation relators can be formed by Player 2, and thus Player 1 wins by traveling completely around the $n$-gon of $m$-gons.

On their first turn, Player 2 can choose $a$, $b$, or $b^{-1}$; however, we note that the choices of $b$ and $b^{-1}$ are symmetric. Thus we may assume Player 2 chooses $a$ or $b$ on their first turn. Then the word $w_2 \equiv_G a^2$ or $ab$. Now for Player 1's next turn, they will choose $b$ if $w_2\equiv_G a^2$ and will choose $a$ if $w_2\equiv_G ab$, yielding the word $w_3 \equiv_G a^2b$ in either case. Note that $w_3$ ends in $ab$, and there is no occurrence of $a^{-1}$ or $b^{-1}$ in the word.

Now suppose for $1 < \ell < n$, that on Player 1's $\ell$-th turn, their play results in the word $w_{2\ell - 1} \equiv_G a^{\ell}b^{\ell -1}$ that ends in $ab$ or $ba$ and contains no occurrence of $a^{-1}$ or $b^{-1}$.\footnote{Note that while $G$ is abelian, in the \textit{word} $w_{2\ell -1}$ the string $ba$ is not equal to the string $ab$} We show that Player 2 cannot win on their $\ell$-th turn and that Player 1 can move to the word $w_{2\ell + 1} \equiv_G a^{\ell+1}b^{\ell}$.

Since the word $w_{2\ell-1}$ ends in $ab$ or $ba$ and $m,n\geq 3$, the relators $a^n $ and $b^m$ cannot be formed by Player 2. 
As per our hypothesis, the word $w_{2\ell - 1}$ has no occurrence of $a^{-1}$ or $b^{-1}$; thus Player 2 cannot form a commutation relator on this turn, either.

Now, since $\ell<n$, Player 2 cannot complete a relator of the form $a^sb^t$, where $n|s$ and $m|t$ on turn $2\ell$. Indeed, since $w_{2\ell-1} \equiv_G a^{\ell}b^{\ell -1}$, Player 2 must choose $b$ to move to $a^{\ell}b^{\ell}$ or $a$ to move to $a^{\ell+1}b^{\ell - 1}$. In the first case,  $n\nmid \ell$ since $\ell<n$. In the second case, if $\ell+1<n$, then $n\nmid \ell +1$. If $\ell + 1 =n$, then $m\nmid \ell -1$ since $n\equiv 1\mod m$. In either case, the relation $a^sb^t = e$ is not satisfied.

We now see that Player 2 cannot possibly win on their $\ell$-th turn. Since a choice of $a^{-1}$ or $b^{-1}$ will result in Player 1 winning on their next turn by means of the commutation relator, Player 2 must choose $a$ or $b$. Then, Player 1's strategy dictates that they will move to the word $w_{2\ell+1} \equiv_G a^{\ell+1}b^{\ell}$. Moreover, in $w_{2 \ell+1}$ there is no occurrence of $a^{-1}$ or $b^{-1}$ and $w_{2 \ell+1}$ must end in either $ab$ or $ba$.

In particular, Player 1 can guarantee reaching $w_{2(n-1)-1} \equiv_G a^{n-1}b^{n-2}$. No matter the choice of Player 2 at this juncture, Player 1 will win on the subsequent turn using the fact that $a^n b^{n-1} \equiv_G e$ or by using the commutation relator.

Now suppose $n\equiv -1\mod m$, so $n= km-1$ for some $k\geq 1$. In this case, Player 1 begins by choosing $b$ and then uses the same strategy as above. This is a winning strategy by an argument similar to the $n \equiv 1 \mod m$ case.

Now suppose $n\equiv 0\mod m$, say $n=km$. Without loss of generality, Player 1 chooses $a$ or $b$ due to symmetry. Player 2 has a winning strategy by choosing $b$ whenever Player 1 chooses $a$ and by choosing $a$ whenever Player 1 chooses $b$. In this case, Player 2 ensures that their play results in the word $w_{2\ell} \equiv_G a^{\ell}b^{\ell}$, which ends in $ab$ or $ba$ and has no occurence of $a^{-1}$ or $b^{-1}$ for $1\leq \ell <n$. To see that Player 1 cannot win on the next turn for $\ell <n$, note that Player 1 cannot complete the relators $a^n$ or $b^m$ since the word $w_{2\ell-1}$ ends in $ab$ or $ba$ and $m,n\geq 3$. They also cannot complete a commutation relator because there is no occurence of $a^{-1}$ or $b^{-1}$. Therefore they must choose $b$ to achieve a word equivalent to $a^{\ell}b^{\ell+1}$ or choose $a$ to achieve a word equivalent to $a^{\ell+1}b^{\ell}$. In the first case, the word is not equivalent to $e$ since $\ell<n$ so $n\nmid \ell$. In the second case, either $\ell + 1<n$, so $n\nmid \ell +1$, or $\ell + 1 = n$. If $\ell + 1 = n$, $m\mid \ell+1$ since $n\equiv 0 \mod m$ and hence $m\nmid \ell$ since $m\geq 3$.

Because of Player 2's strategy, Player 2 can ensure reaching $w_{2(n-1)} \equiv_G a^{n-1}b^{n-1} \equiv_G a^{-1}b^{-1}$. Player 2 then wins on the subsequent move either at $e$ since $a^n b^n = e$ or by using the commutation relator.

\end{proof}

\begin{figure}[hbt]
    \centering
    \begin{tikzpicture}[scale=0.75]

        \node (a3) at (-3, 0.5) [vert] {};
        \node at (-3, 1.0) {$a^3$};
        \node (a3b2) at (-4.2,-1.4) [vert] {}; 
        \node at (-4.8,-1.8) {$a^3b^2$};
        \node (a3b) at (-1.8,-1.4) [vert] {}; 
        \node at (-1.8,-1.8) {$a^3b$};
        
        \path[black-r] (a3) to (a3b);
        \path[black-r] (a3b) to (a3b2);
        \path[black-r] (a3b2) to (a3);
        
        \node (b2) at (-1.2, 1.9) [vert] {};
        \node at (-1.2,1.5)  {$b^2$};
        \node (e) at (0, 3.8) [vert] {};
        \node at (0, 4.2) {$e$};
        \node (b) at (1.2, 1.9) [vert] {};
        \node at (1.2, 1.5) {$b$};
        
        \path[black-r] (e) to (b);
        \path[black-r] (b) to (b2);
        \path[black-r] (b2) to (e);

        \node (ab) at (4.2, -1.4) [vert] {};
        \node at (4.2, -1.8) {$ab$};
        \node (ab2) at (1.8, -1.4) [vert] {};
        \node at (1.8, -1.8) {$ab^2$};
        \node (a) at (3, 0.5) [vert] {};
        \node at (3.2, 0.9) {$a$};
        
        \path[black-r] (a) to (ab);
        \path[black-r] (ab) to (ab2);
        \path[black-r] (ab2) to (a);
      
        \node (a2b) at (1.2, -4.7) [vert] {};
        \node at (1.2, -5.1) {$a^2b$};
        \node (a2b2) at (-1.2, -4.7) [vert] {};
        \node at (-1.2, -5.1) {$a^2b^2$};
        \node (a2) at (0, -2.8) [vert] {};
        \node at (-0.2, -2.3) {$a^2$};
        
        \path[black-r] (a2) to (a2b);
        \path[black-r] (a2b) to (a2b2);
        \path[black-r] (a2b2) to (a2);
        
        \path[r, -stealth, bend left] (e) to (a);
        \path[b, -stealth, bend right] (a) to (a2);
        \path[r, -stealth] (a2) to (a2b);
        \path[b, -stealth] (a2b) to (a2b2);
        \path[r, -stealth, bend left] (a2b2) to (a3b2);
        \path[b, -stealth] (a3b2) to (a3);
        
    \end{tikzpicture}
    \caption{A partial Cayley graph for $\Z_4 \times \Z_3$ - a 4-gon of 3-gons. The colored edges give an example of Player 1's strategy as described in \cref{theorem: Zn-Zm mod pm1 and 0}. The `{\color{red}red}' edges denote Player 1 moves and the `{\color{blue}blue}' edges denote Player 2 moves. By choosing the generator opposite to Player 2's last move, Player 1 will always ensure victory. Note as well that, in this example, the game word is currently equivalent to $a^3b^3 = a^3$, but Player 2 has not achieved a relator.}
    \label{fig: for thm 4.1}
\end{figure}
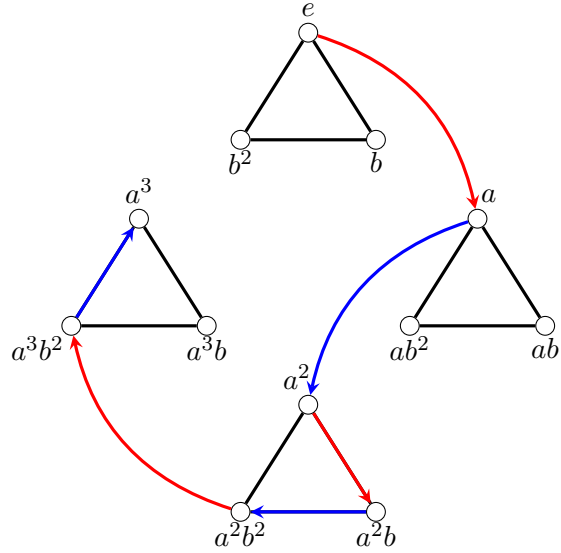

\begin{remark}\label{cor: Zn x Zm where m,n=2}

By \cref{thm: 2 player dihedral REL}, and the fact that the undirected Cayley graphs of $(D_{n}, \{r,s\})$ and $(\Z_n\times \Z_2, \{(1,0),(0,1)\})$ are isomorphic for $n\geq 3$, we also know the winner of $\REL(\Z_n \times \Z_2)$ for $n \geq 3$. Additionally, we know Player 2 has a winning strategy $\REL(\Z_2 \times \Z_2)$ since its undirected Cayley graph is equal to that of $(\Z_4, \{1\})$.
\end{remark}

Note that for $m\geq 4$, there are cases not covered by \cref{theorem: Zn-Zm mod pm1 and 0}. However, a small additional argument can be made to cover all cases when $m=4$, namely that Player 2 has a winning strategy when $n\equiv 2\mod 4$.
\begin{proposition} \label{cor: Zn x Z4}
For $\REL(\Z_n \times \Z_4, \{(1,0), (0,1)\})$, with $n \geq 3$, we have 
    \begin{enumerate}
        \item Player 1 has a winning strategy if $n \equiv \pm 1 \mod 4$.
        \item Player 2 has a winning strategy if $n \equiv 0 \mod 4$ or $n \equiv 2 \mod 4$.
    \end{enumerate}
\end{proposition}

\begin{proof} Let $G \cong \Z_n \times \Z_4$ and let $a$ and $b$ denote the generators for $G$, with $a^n=e$ and $b^4=e$.

By \Cref{theorem: Zn-Zm mod pm1 and 0} the only case left to prove is that of $n \equiv 2 \mod 4$. Let $n = 4k+2$, with $k \geq 1$. We describe a winning strategy for Player 2 dependent on Player 1 choosing $a$ or $b$ for their first turn. The arguments for Player 1 choosing $a^{-1}$ or $b^{-1}$ is symmetric.

Suppose Player 1 chooses $a$. Player 2's strategy is to ensure that the word $w_{2\ell} \equiv_G a^sb^t$ has the property that $s-t=2$. Hence, for their first turn, Player 2 will choose $a$. 
On subsequent turns, Player 2 will choose $a$ if Player 1 chose $b$ and will choose $b$ if Player 1 chose $a$. 
By a similar argument to that in the proof of \cref{theorem: Zn-Zm mod pm1 and 0}, Player 2 will win on turn $\ell =8k+2$ at $a^{4k+2}b^{4k}$.

Now suppose Player 1 chooses $b$ for their first turn. Then Player 2's strategy here is to ensure that $w_{2\ell} \equiv_G a^sb^t$ has $t-s=2$. Hence, Player 2 will choose $b$ for their first turn and then will choose the opposing generator to Player 1's choice on their subsequent turns.
Player 2 then wins on turn $\ell = 8k+4$ at $a^{4k+2}b^{4k+2}$.
\end{proof}

\section{Three-Player $\REL$ for Dihedral Groups}
\label{sec: three-player REL}
In this section, we examine the $\REL$ game for dihedral groups with three players.
We first modify the definition of $\REL$ to include $n$ players as first done by Benesh and Gaetz in \cite{BG18}. The first player to complete a relator still wins the game. The following player is runner-up. The next player is third, and so on. This ranking is important because it gives each player a preference for who wins. If a player cannot ensure victory for themselves, then they will assist the player who ensures their highest possible ranking to win the game. 

For three players, this is a bit simpler. Since there are only two opponents for a given player, if that player cannot win, then they will play to help their preferred opponent win so that they finish second instead of last. 

 When referring to the \textit{Relator Achievement Game for $n$ Players} for a group $G$ and generating set $S$, we use the notation $\REL_n(G,S)$.

\begin{remark}\label{rem: 3 player no third edge}
Note that \Cref{rem: no third edge} still holds. A player will never prefer to be last; completing the third edge of a square leads to the next player winning and thus a last place finish. Thus, no player will complete a third edge of a square if it can be avoided. Due to this, \cref{rem: $s$ guarantees next 2 moves} still holds. That is, a choice of the generator $s$ forces the following two moves.
\end{remark}

\begin{theorem} \label{thm:3-player REL}
Player 1 has a winning strategy for $\REL_3(D_n, \{r,s\})$ if $n$ is odd, and Player 3 has a winning strategy if $n$ is even.
\end{theorem}
\begin{proof}
First, we suppose that $n$ is odd. We will consider the cases where $n\equiv 1\mod 4$ and $n\equiv 3\mod 4$, separately. In either case, we will show that Player 1 has a winning strategy by choosing the generator $s$ every turn until a winning move is available. The key ingredient to this proof is \Cref{rem: 3 player no third edge}. As in the proof of \Cref{lem: 3 forward 1 over} we assume, without loss of generality, that players choose words equivalent to $r^is^j$ with $0\leq i\leq n-1$, $0\leq j\leq 1$, where $i$ is non-decreasing. 

Suppose $n = 4k+1$ for some integer $k\geq 1$. We first note that Player 1 ensures victory by choosing $s$ to move from a word equivalent to $r^{n-3}s$ to a word equivalent to $r^{n-3}$. Indeed, by \Cref{rem: 3 player no third edge}, Player 2 is forced to choose $r$ to move to $r^{n-2}$ and Player 3 must choose $r$ or $s$ to move to a losing position in $r^{n-1}$ or $r^{n-2}s$, whereby Player 1 wins at $e$ or $r^{n-3}s$, respectively.

Due to Player 1's strategy and \cref{rem: 3 player no third edge}, we note that Player 1 can guarantee moving from $e$ to $s$, $r^2s$ to $r^2$, and, in general, from $r^{\ell}$ to $r^{\ell}s$ where $\ell\equiv 0\mod 4$ or $r^m s$ to $r^m$ where $m \equiv 2\mod 4$. Since $n-3 = 4k-2\equiv 2\mod 4$, Player 1 guarantees moving from $r^{n-3}s$ to $r^{n-3}$ and thus wins as described above.

Now suppose $n=4k+3$ for some integer $k \geq 1$. In this case, we note that Player 1 ensures victory by moving from $r^{n-3}$ to $r^{n-3}s$. By \Cref{rem: 3 player no third edge}, Player 2 is forced to choose $r^{-1}$ to move to $r^{n-2}s$ and Player 3 must choose $r^{-1}$ or $s$ to move to a losing position in $r^{n-1}s$ or $r^{n-2}$, whereby Player 1 wins at $s$ or $r^{n-3}$, respectively. In general, Player 1 guarantees moving to the same words as in the $n = 4k+1$ case. However, in this case $n-3 = 4k \equiv 0\mod 4$, so Player 1 guarantees moving from $r^{n-3}$ to $r^{n-3}s$ and thus wins.

If $n$ is even, then the preference of a non-winning player is important in deciding the game. We show that Player 1 can never win and thus will help Player 3 win in order to finish second instead of finishing third. We consider the cases where $n\equiv 2\mod 4$ and $n\equiv 0\mod 4$ separately.

Suppose $n = 4k+2$ for some integer $k \geq 1$. We begin by showing that Player 1 will always finish last if they ever choose $r$ or $r^{-1}$. On the other hand, if Player 1 always chooses $s$, they can guarantee victory for Player 3, thus securing a runner-up finish, which is preferable to finishing last. Thus Player 3 will win in this case due to the preferences of Player 1.

We first consider the case where Player 1 begins by choosing $r$ or $r^{-1}$. In this case, Player 2 has a winning strategy similar to the odd cases by choosing $s$ successively until the winning move. Player 2 guarantees moving to words equivalent to $r^{\ell}s$ where $\ell\equiv 1\mod 4$ and $r^m$ where $m\equiv 3 \mod 4$ by \cref{rem: 3 player no third edge}. Since $n-3 = 4k-1 \equiv 3\mod 4$, Player 2 guarantees moving from $r^{n-3}s$ to $r^{n-3}$, which is a winning move.

We now consider the case where Player 1 begins by choosing $s$ for some number of turns, but then chooses $r$ or $r^{-1}$ on a later turn. That is, upon choosing $r$ or $r^{-1}$, Player 1 moves from $r^{\ell}$ to $r^{\ell + 1}$ via $r$ where $\ell\equiv 0\mod 4$ or from $r^ms$ to $r^{m+1}s$ via $r^{-1}$ where $m\equiv 2\mod 4$. Player 2 therefore, by always choosing $s$ after, moves to $r^{\ell+1}s$ where $\ell\equiv 0\mod 4$ and $r^{m+1}$ where $m\equiv 2\mod 4$. Since $n-3\equiv 3\mod 4$, Player 2 guarantees moving from $r^{n-3}s$ to $r^{n-3}$, which is a winning move.

Having shown that Player 1 finishes last if they do not always choose $s$, we now assume Player 1 does always choose $s$. Then Player 1 ensures that Player 3 wins by moving from $r^{n-2}$ to $r^{n-2}s$, thereby forcing Player 2 to move from $r^{n-2}s$ to $r^{n-1}s$ or to $r^{n-3}s$. Thereafter, Player 3 wins at $s$ or at $r^{n-3}$. Since $n \equiv 2\mod 4$, as shown in previous cases, Player 1 guarantees arriving at $r^{n-2}$ from $r^{n-2}s$.

Now suppose $n=4k$ for some integer $k \geq 1$. We show by exhaustion that Player 1 can never win. However, they can ensure a runner-up finish via a Player 3 win.

There are four losing cases for Player 1:
\begin{enumerate}
    \item[(1)] If Player 1 always chooses $s$, then Player 3 has a winning strategy.
    \item[(2)] If Player 1 starts with $s$ but later chooses $r$ or $r^{-1}$, then Player 2 has a winning strategy.
    \item[(3)]  If Player 1 starts with $r$ or $r^{-1}$ and Player 2 always chooses $s$, then Player 3 has a winning strategy.
    \item[(4)] If Player 1 starts with $r$ or $r^{-1}$ and Player 2 chooses $r$ or $r^{-1}$ at some point, then Player 3 has a winning strategy.
\end{enumerate}

(1) Using \Cref{rem: 3 player no third edge} as in the previous cases, if Player 1 chooses $s$ every turn, then Player 1 guarantees moving to $r^{n-2}$ from $r^{n-2}s$ since $n-2\equiv 2 \mod 4$. Player 2 is forced to move to $r^{n-1}$ or $r^{n-3}$, and Player 3 wins at $e$ or $r^{n-3}s$.

(2) Suppose Player 1 begins by choosing $s$ for some number of turns, then later chooses $r$ or $r^{-1}$. As described in the $n=4k+2$ case, Player 1 moves to $r^{\ell}$ where $\ell\equiv 1\mod 4$ or $r^ms$ where $m\equiv 3 \mod 4$. Player 2 then moves to $r^{\ell}s$ or $r^m$ by choosing $s$. Since $n-3\equiv 1 \mod 4$, Player 2 guarantees moving from $r^{n-3}$ to $r^{n-3}s$, which is a winning move.

(3) Without loss of generality, suppose Player 1 starts with $r$ and Player 2 chooses $s$ at each turn. Then Player 2 moves to $r^{\ell}s$ where $\ell\equiv 1\mod 4$ and $r^m$ where $m\equiv 3\mod 4$. In particular, Player 2 will move to $r^{4k-1}$, thus ensuring that Player 3 wins on the next turn. 

(4) Without loss, suppose Player 1 opens with $r$. Then Player 2 must choose $r$ or $r^{-1}$ at some point (if they hope to avoid finishing last, by case (3) above). A choice of $r$ by Player 2 will move to $r^{\ell}$ where $\ell\equiv 2\mod 4$ while a choice of $r^{-1}$ will move to $r^ms$ where $m\equiv 0\mod 4$. Player 3 then has a winning strategy by choosing $s$ unless another move wins the game. Using this strategy, Player 3 moves to $r^{\ell}s$ and $r^m$ where $\ell\equiv 2\mod 4$ and $m\equiv 0 \mod 4$. Since $n-2\equiv 2\mod 4$, Player 3 guarantees moving to $r^{n-2}s$. This forces Player 1 to move to $r^{n-1}s$ and Player 2 to move to either $s$ or $r^{n-1}$, whereby Player 3 wins at $e$.
\end{proof}

As one can see from the proof of \cref{thm:3-player REL}, the game is considerably more complex after introducing just one more player. 
This complexity lies in the fact that there is no extension of \cref{rem: no third edge} that allows a player to force more than the next two moves. Hence, it is unclear how a player's move will affect their following move unless we consider the preferences of all players.

\section{Open Questions}
\label{sec: open questions}

\begin{itemize}
 \item When first devising the games $\REL$ and $\RAV$, we wanted to create a combinatorial game that utilized the Cayley graph of a group. Although the Cayley graph is not necessary in defining our relator games, we have found it useful when constructing some of our proofs. To that end, one can study the \textit{make a cycle} and \textit{avoid a cycle} games on general graphs. We expect these games to be more challenging to study due to the absence of properties such as graph regularity and symmetry inherent in Cayley graphs.

    \item A fundamental problem in combinatorial game theory for impartial games is to find the nim-number of a game (see \cite{Sie13}). These allow one to determine the outcome of the game as well as of game sums. While we have determined the outcome of the games $\REL$ and $\RAV$ for several families of groups, we leave open the problem of computing their nim-numbers.
    
    \item Another goal is to extend results on $\REL$ and $\RAV$ to $\REL_n$ and $\RAV_n$, where more than two players are allowed. For the dihedral groups, this becomes difficult after more than three players are involved since a player can only force moves two ahead and thus loses control over their future moves.
    Note that the related games $\GEN$ and $\DNG$ for $n$-players have been studied in \cite{BG18}. 
    
    \item One can of course ask for the outcomes of $\REL$ and $\RAV$ on other families of finite groups. Of specific interest are the generalized dihedral groups (see \cref{ex: semi-direct and generalized dihedral}). We have results for $\RAV$ for generalized dihedral groups via \cref{thm: RAV any group with order two}.
    For a finite generalized dihedral group $G \cong H \rtimes \Z_2$, suppose we have a winning strategy for $\REL(H,T)$. Can we then determine a winning strategy for $\REL(G,S)$ in a manner similar to that of \cref{thm: RAV any group with order two}? 
    
    \item In computing several game trees while working through examples, we observed that most games of $\REL$ end after traversing at most half the vertices in the Cayley graph. We have also observed that the game seems to be less complex when more generators are involved. These lead to interesting questions from a computational point of view. Can one find winning strategies utilizing a minimal number of moves or find a correlation between sparseness of the Cayley graph and computational complexity of the game?
    
    \item One can certainly explore the games $\REL$ and $\RAV$ via a computer program. The authors have done some preliminary work on this with College of Wooster undergraduates Minhwa Lee and Pavithra Brahmananda Reddy on this project \footnote{mlee21@wooster.edu, pbrahmanandareddy22@wooster.edu. The second author would like to thank the College of Wooster Sophomore Research Program for helping finance Minhwa Lee's and Pavithra Brahmananda Reddy's work.}. One avenue is to apply machine learning techniques such as reinforcement learning to create an A.I. for different families of groups.
    
    \item Both authors have incorporated Cayley graphs into their abstract algebra courses. Using the games of $\REL$ and $\RAV$ for Cayley graphs of dihedral groups and symmetric groups is an alternative way of getting students to practice understanding the structure of these groups. A structured and rigorous implementation of such an approach is a possible direction for pedagogical research.
\end{itemize}

\bibliography{CayleyGames.bib}

\begin{thebibliography}{10}

\bibitem{cycleGame}
Ryan Alvarado, Maia Averett, Benjamin Gaines, Christopher Jackson, Mary~Leah
  Karker, Malgorzata~Aneta Marciniak, Francis Su, and Shanise Walker.
\newblock The game of cycles, 2020.

\bibitem{AH87}
M.~Anderson and F.~Harary.
\newblock Achievement and avoidance games for generating abelian groups.
\newblock {\em Internat. J. Game Theory}, 16(4):321--325, 1987.

\bibitem{BES16-symalt}
Bret~J. Benesh, Dana~C. Ernst, and N\'{a}ndor Sieben.
\newblock Impartial avoidance and achievement games for generating symmetric
  and alternating groups.
\newblock {\em Int. Electron. J. Algebra}, 20:70--85, 2016.

\bibitem{BES16-1stPub}
Bret~J. Benesh, Dana~C. Ernst, and N\'{a}ndor Sieben.
\newblock Impartial avoidance games for generating finite groups.
\newblock {\em North-West. Eur. J. Math.}, 2:83--103, 2016.

\bibitem{BES17}
Bret~J. Benesh, Dana~C. Ernst, and N\'{a}ndor Sieben.
\newblock Impartial achievement games for generating generalized dihedral
  groups.
\newblock {\em Australas. J. Combin.}, 68:371--384, 2017.

\bibitem{BES19}
Bret~J. Benesh, Dana~C. Ernst, and N\'{a}ndor Sieben.
\newblock Impartial achievement games for generating nilpotent groups.
\newblock {\em J. Group Theory}, 22(3):515--527, 2019.

\bibitem{BG18}
Bret~J. Benesh and Marisa~R. Gaetz.
\newblock A {$q$}-player impartial avoidance game for generating finite groups.
\newblock {\em Internat. J. Game Theory}, 47(2):451--461, 2018.

\bibitem{Frankl87}
P.~Frankl.
\newblock On a pursuit game on {C}ayley graphs.
\newblock {\em Combinatorica}, 7(1):67--70, 1987.

\bibitem{Leh19}
Florian Lehner.
\newblock Firefighting on trees and {C}ayley graphs.
\newblock {\em Australas. J. Combin.}, 75:66--72, 2019.

\bibitem{NW83}
Richard Nowakowski and Peter Winkler.
\newblock Vertex-to-vertex pursuit in a graph.
\newblock {\em Discrete Math.}, 43(2-3):235--239, 1983.

\bibitem{Quill78}
A.~Quilliot.
\newblock Jeux et pointes fixes sur les graphes.
\newblock {\em Th\`{e}se de 3\`{e}me cycle, \emph{Universit\'{e} de Paris VI}},
  pages 131--145, 1978.

\bibitem{Sie13}
Aaron~N. Siegel.
\newblock {\em Combinatorial game theory}, volume 146 of {\em Graduate Studies
  in Mathematics}.
\newblock American Mathematical Society, Providence, RI, 2013.

\bibitem{Su20}
Francis Su.
\newblock {\em Mathematics for human flourishing}.
\newblock Yale University Press, New Haven, CT, [2020] \copyright 2020.
\newblock With reflections by Christopher Jackson.

\end{thebibliography}
\bibliographystyle{plain}

\Addresses

\end{document}